\documentclass[a4paper,10pt]{amsart}
\usepackage{a4wide}
\usepackage{enumerate}

\title{A refined bijection between alternating permutations and 0-1-2 increasing trees}

\author{Heesung Shin}
\address[Heesung Shin]{Universit\'{e} de Lyon; Universit\'{e} Lyon 1; Institut Camille Jordan; UMR 5208 du CNRS; 43, boulevard du 11 novembre 1918, F-69622 Villeurbanne Cedex, France}
\email{hshin@math.univ-lyon1.fr}

\date{\today}

\usepackage{color}
\usepackage{ifpdf}
\ifpdf 
  \usepackage[pdftex]{graphicx}
  \DeclareGraphicsExtensions{.pdf,.png,.mps}
  \usepackage{pgf}
  \usepackage{tikz}
\else 
  \usepackage{graphicx}
  \DeclareGraphicsExtensions{.eps,.bmp}
  \usepackage{pgf}
  \usepackage{tikz}
  \usepackage{pstricks}
\fi
\usepackage{epic}
\usepackage{wrapfig}

\newtheorem{thm}{Theorem}

\theoremstyle{definition}

\newtheorem*{rmk}{Remark}
\newtheorem*{ex}{Example}

\DeclareMathOperator\inv{inv}
\DeclareMathOperator\threeonetwo{31-2}

\newcommand\set[1]{\left\{#1\right\}}
\newcommand\abs[1]{\left|#1\right|}

\def\A{\mathcal{A}}
\def\An{\mathcal{A}_n}
\def\Ank{\mathcal{A}_{n,k}}
\def\T{\mathcal{T}}
\def\Tn{\mathcal{T}_n}
\def\Tnk{\mathcal{T}_{n,k}}
\def\312{\threeonetwo}

\begin{document}
\maketitle

\begin{abstract}
We construct a refined bijection $\varphi$ between alternating permutations and 0-1-2 increasing trees with degree at most $2$. It satisfies that the first element of alternating permutation $\pi$ is equal to the first vertex in $\varphi(\pi)$ in the postorder.
\end{abstract}

\section{Introduction}

Let $\An$ be the set of {\em alternating permutations} $\pi=\pi_1\pi_2\dots\pi_n$ on $[n]:=\set{1,2,\dots,n}$ satisfying
$$\pi_1 > \pi_2 < \pi_3 > \pi_4 < \cdots.$$
Let $\Ank = \set{\pi\in\An : \pi_1=k}$.
Let $\Tn$ be the set of {\em 0-1-2 increasing trees} $T$ on $[n]$ with degree at most 2. Given a increasing tree $T$, the unique path towards a smallest child of each vertex from the root $1$ is called the {\em main chain} of $T$ and we denote the leaf at the end of the main chain by $p(T)$.
Let $\Tnk = \set{T\in\Tn : p(T)=k}$.

It is important that $\An$ and $\Tn$ are equinumerous (see \cite{Foa72,FS74}, and also the bijection in \cite{FS73,Don75}).
Also $\Ank$ and $\Tnk$ are equinumerous~\cite{Pou82, KPP94}.
However, no one has managed to establish explicit bijections between $\Ank$ and $\Tnk$, as mentioned in the paper of
Kuznetsov, Pak and Postnikov \cite{KPP94}, who asked also for a combinatorial explanation of
the identity
$$\abs{\Tnk} = \abs{\T_{n,k-1}} + \abs{\T_{n-1,n-k+1}}$$
in the model of 0-1-2 increasing trees.
The aim of this paper is to give a bijection between $\Ank$ and $\Tnk$. To the best knowledge
of the author, this is the first bijection between these two sets.

\section{A bijection between $\Ank$ and $\Tnk$}
For $n=1$ or $2$, since $\abs{\A_n}=\abs{\T_n}=1$, we can define trivially $\varphi:\Ank\to \Tnk$.
For $n\ge 3$, given $\pi \in \Ank$ ($k=\pi_1$), we define the mapping $\varphi:\Ank \to \Tnk$ recursively as follows:
\begin{enumerate}[(a)]
\item \label{case:a} If $\pi_2=k-1$, then define $\pi'\in \A_{n-2,i-2}$ by deleting $k-1$ and $k$ from $\pi$ and relabeling by $[n-2]$ where $i> k$. We get $T'=\varphi(\pi') \in \T_{n-2,i-2}$.
    Relabel $T'$ by $\set{1,\dots,k-2,k+1,\dots,n}$, denoted by $T''$. Let $m$ be the minimal vertex greater than $k$ in the main chain of $T''$ and $j$ the parent of $m$ in $T''$. Then insert a vertex $k-1$ in the middle of the edge $(j,m)$ and add the edge $(k,k-1)$.
    $$
    
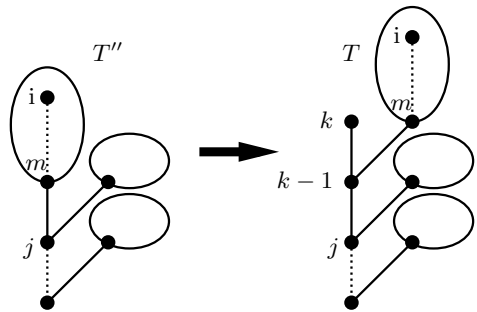
\centering
\begin{pgfpicture}{38.00mm}{29.20mm}{122.00mm}{73.20mm}
\pgfsetxvec{\pgfpoint{0.80mm}{0mm}}
\pgfsetyvec{\pgfpoint{0mm}{0.80mm}}
\color[rgb]{0,0,0}\pgfsetlinewidth{0.30mm}\pgfsetdash{}{0mm}
\pgfsetlinewidth{1.20mm}\pgfmoveto{\pgfxy(95.00,65.00)}\pgflineto{\pgfxy(105.00,65.00)}\pgfstroke
\pgfmoveto{\pgfxy(105.00,65.00)}\pgflineto{\pgfxy(102.20,65.70)}\pgflineto{\pgfxy(102.20,64.30)}\pgflineto{\pgfxy(105.00,65.00)}\pgfclosepath\pgffill
\pgfmoveto{\pgfxy(105.00,65.00)}\pgflineto{\pgfxy(102.20,65.70)}\pgflineto{\pgfxy(102.20,64.30)}\pgflineto{\pgfxy(105.00,65.00)}\pgfclosepath\pgfstroke
\pgfputat{\pgfxy(120.00,80.00)}{\pgfbox[bottom,left]{\fontsize{9.10}{10.93}\selectfont \makebox[0pt]{$T$}}}
\pgfsetdash{{0.30mm}{0.50mm}}{0mm}\pgfsetlinewidth{0.30mm}\pgfmoveto{\pgfxy(120.00,50.00)}\pgflineto{\pgfxy(120.00,40.00)}\pgfstroke
\pgfcircle[fill]{\pgfxy(120.00,50.00)}{0.80mm}
\pgfsetdash{}{0mm}\pgfcircle[stroke]{\pgfxy(120.00,50.00)}{0.80mm}
\pgfcircle[fill]{\pgfxy(120.00,40.00)}{0.80mm}
\pgfcircle[stroke]{\pgfxy(120.00,40.00)}{0.80mm}
\pgfmoveto{\pgfxy(120.00,50.00)}\pgflineto{\pgfxy(130.00,60.00)}\pgfstroke
\pgfcircle[fill]{\pgfxy(120.00,70.00)}{0.80mm}
\pgfcircle[stroke]{\pgfxy(120.00,70.00)}{0.80mm}
\pgfmoveto{\pgfxy(120.00,40.00)}\pgflineto{\pgfxy(130.00,50.00)}\pgfstroke
\pgfcircle[fill]{\pgfxy(130.00,50.00)}{0.80mm}
\pgfcircle[stroke]{\pgfxy(130.00,50.00)}{0.80mm}
\pgfputat{\pgfxy(117.00,69.00)}{\pgfbox[bottom,left]{\fontsize{9.10}{10.93}\selectfont \makebox[0pt][r]{$k$}}}
\pgfputat{\pgfxy(117.00,59.00)}{\pgfbox[bottom,left]{\fontsize{9.10}{10.93}\selectfont \makebox[0pt][r]{$k-1$}}}
\pgfcircle[fill]{\pgfxy(120.00,60.00)}{0.80mm}
\pgfcircle[stroke]{\pgfxy(120.00,60.00)}{0.80mm}
\pgfmoveto{\pgfxy(120.00,50.00)}\pgflineto{\pgfxy(120.00,60.00)}\pgfstroke
\pgfmoveto{\pgfxy(130.00,70.00)}\pgflineto{\pgfxy(120.00,60.00)}\pgfstroke
\pgfmoveto{\pgfxy(120.00,70.00)}\pgflineto{\pgfxy(120.00,60.00)}\pgfstroke
\pgfellipse[stroke]{\pgfxy(130.00,79.49)}{\pgfxy(6.00,0.00)}{\pgfxy(0.00,9.51)}
\pgfcircle[fill]{\pgfxy(130.00,70.00)}{0.80mm}
\pgfcircle[stroke]{\pgfxy(130.00,70.00)}{0.80mm}
\pgfputat{\pgfxy(130.00,72.00)}{\pgfbox[bottom,left]{\fontsize{9.10}{10.93}\selectfont \makebox[0pt][r]{$m$}}}
\pgfputat{\pgfxy(118.00,48.00)}{\pgfbox[bottom,left]{\fontsize{9.10}{10.93}\selectfont \makebox[0pt][r]{$j$}}}
\pgfcircle[fill]{\pgfxy(130.00,60.00)}{0.80mm}
\pgfcircle[stroke]{\pgfxy(130.00,60.00)}{0.80mm}
\pgfellipse[stroke]{\pgfxy(133.50,63.49)}{\pgfxy(6.50,0.00)}{\pgfxy(0.00,4.51)}
\pgfellipse[stroke]{\pgfxy(133.50,53.49)}{\pgfxy(6.50,0.00)}{\pgfxy(0.00,4.51)}
\pgfcircle[fill]{\pgfxy(130.00,84.00)}{0.80mm}
\pgfcircle[stroke]{\pgfxy(130.00,84.00)}{0.80mm}
\pgfsetdash{{0.30mm}{0.50mm}}{0mm}\pgfmoveto{\pgfxy(130.00,84.00)}\pgflineto{\pgfxy(130.00,70.00)}\pgfstroke
\pgfputat{\pgfxy(128.00,83.00)}{\pgfbox[bottom,left]{\fontsize{9.10}{10.93}\selectfont \makebox[0pt][r]{i}}}
\pgfputat{\pgfxy(80.00,80.00)}{\pgfbox[bottom,left]{\fontsize{9.10}{10.93}\selectfont \makebox[0pt]{$T''$}}}
\pgfmoveto{\pgfxy(70.00,50.00)}\pgflineto{\pgfxy(70.00,40.00)}\pgfstroke
\pgfcircle[fill]{\pgfxy(70.00,50.00)}{0.80mm}
\pgfsetdash{}{0mm}\pgfcircle[stroke]{\pgfxy(70.00,50.00)}{0.80mm}
\pgfcircle[fill]{\pgfxy(70.00,40.00)}{0.80mm}
\pgfcircle[stroke]{\pgfxy(70.00,40.00)}{0.80mm}
\pgfmoveto{\pgfxy(70.00,50.00)}\pgflineto{\pgfxy(80.00,60.00)}\pgfstroke
\pgfmoveto{\pgfxy(70.00,40.00)}\pgflineto{\pgfxy(80.00,50.00)}\pgfstroke
\pgfcircle[fill]{\pgfxy(80.00,50.00)}{0.80mm}
\pgfcircle[stroke]{\pgfxy(80.00,50.00)}{0.80mm}
\pgfmoveto{\pgfxy(70.00,50.00)}\pgflineto{\pgfxy(70.00,60.00)}\pgfstroke
\pgfellipse[stroke]{\pgfxy(70.00,69.49)}{\pgfxy(6.00,0.00)}{\pgfxy(0.00,9.51)}
\pgfcircle[fill]{\pgfxy(70.00,60.00)}{0.80mm}
\pgfcircle[stroke]{\pgfxy(70.00,60.00)}{0.80mm}
\pgfputat{\pgfxy(70.00,62.00)}{\pgfbox[bottom,left]{\fontsize{9.10}{10.93}\selectfont \makebox[0pt][r]{$m$}}}
\pgfputat{\pgfxy(68.00,48.00)}{\pgfbox[bottom,left]{\fontsize{9.10}{10.93}\selectfont \makebox[0pt][r]{$j$}}}
\pgfcircle[fill]{\pgfxy(80.00,60.00)}{0.80mm}
\pgfcircle[stroke]{\pgfxy(80.00,60.00)}{0.80mm}
\pgfellipse[stroke]{\pgfxy(83.50,63.49)}{\pgfxy(6.50,0.00)}{\pgfxy(0.00,4.51)}
\pgfellipse[stroke]{\pgfxy(83.50,53.49)}{\pgfxy(6.50,0.00)}{\pgfxy(0.00,4.51)}
\pgfcircle[fill]{\pgfxy(70.00,74.00)}{0.80mm}
\pgfcircle[stroke]{\pgfxy(70.00,74.00)}{0.80mm}
\pgfsetdash{{0.30mm}{0.50mm}}{0mm}\pgfmoveto{\pgfxy(70.00,74.00)}\pgflineto{\pgfxy(70.00,60.00)}\pgfstroke
\pgfputat{\pgfxy(68.00,73.00)}{\pgfbox[bottom,left]{\fontsize{9.10}{10.93}\selectfont \makebox[0pt][r]{i}}}
\end{pgfpicture}%
$$
    We get the tree $T = \varphi(\pi) \in \Tnk$.
\item \label{case:b} If $\pi_2<k-1$, then define $\pi'=(k-1~k)\pi \in \A_{n,k-1}$ (exchange $k-1$ and $k$ in $\pi$). We get $T'=\varphi(\pi') \in \T_{n,k-1}$.
    \begin{enumerate}[(1)]
    \item If $k$ is a sibling of $k-1$ in $T'$, then we get the tree $T = \varphi(\pi) \in \Tnk$ modifying as follows: 
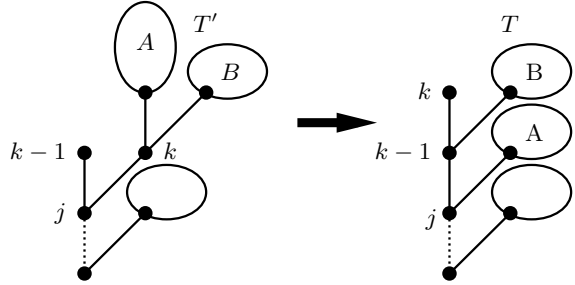
\label{case:b1}
    $$
    \centering
\begin{pgfpicture}{38.00mm}{29.20mm}{122.00mm}{70.00mm}
\pgfsetxvec{\pgfpoint{0.80mm}{0mm}}
\pgfsetyvec{\pgfpoint{0mm}{0.80mm}}
\color[rgb]{0,0,0}\pgfsetlinewidth{0.30mm}\pgfsetdash{}{0mm}
\pgfputat{\pgfxy(130.00,80.00)}{\pgfbox[bottom,left]{\fontsize{9.10}{10.93}\selectfont \makebox[0pt]{$T$}}}
\pgfsetdash{{0.30mm}{0.50mm}}{0mm}\pgfmoveto{\pgfxy(120.00,50.00)}\pgflineto{\pgfxy(120.00,40.00)}\pgfstroke
\pgfcircle[fill]{\pgfxy(120.00,50.00)}{0.80mm}
\pgfsetdash{}{0mm}\pgfcircle[stroke]{\pgfxy(120.00,50.00)}{0.80mm}
\pgfcircle[fill]{\pgfxy(120.00,40.00)}{0.80mm}
\pgfcircle[stroke]{\pgfxy(120.00,40.00)}{0.80mm}
\pgfmoveto{\pgfxy(120.00,50.00)}\pgflineto{\pgfxy(130.00,60.00)}\pgfstroke
\pgfcircle[fill]{\pgfxy(120.00,70.00)}{0.80mm}
\pgfcircle[stroke]{\pgfxy(120.00,70.00)}{0.80mm}
\pgfmoveto{\pgfxy(120.00,40.00)}\pgflineto{\pgfxy(130.00,50.00)}\pgfstroke
\pgfcircle[fill]{\pgfxy(130.00,50.00)}{0.80mm}
\pgfcircle[stroke]{\pgfxy(130.00,50.00)}{0.80mm}
\pgfputat{\pgfxy(117.00,69.00)}{\pgfbox[bottom,left]{\fontsize{9.10}{10.93}\selectfont \makebox[0pt][r]{$k$}}}
\pgfputat{\pgfxy(117.00,59.00)}{\pgfbox[bottom,left]{\fontsize{9.10}{10.93}\selectfont \makebox[0pt][r]{$k-1$}}}
\pgfcircle[fill]{\pgfxy(120.00,60.00)}{0.80mm}
\pgfcircle[stroke]{\pgfxy(120.00,60.00)}{0.80mm}
\pgfmoveto{\pgfxy(120.00,50.00)}\pgflineto{\pgfxy(120.00,60.00)}\pgfstroke
\pgfmoveto{\pgfxy(130.00,70.00)}\pgflineto{\pgfxy(120.00,60.00)}\pgfstroke
\pgfmoveto{\pgfxy(120.00,70.00)}\pgflineto{\pgfxy(120.00,60.00)}\pgfstroke
\pgfellipse[stroke]{\pgfxy(133.50,73.49)}{\pgfxy(6.50,0.00)}{\pgfxy(0.00,4.51)}
\pgfputat{\pgfxy(134.00,72.00)}{\pgfbox[bottom,left]{\fontsize{9.10}{10.93}\selectfont \makebox[0pt]{B}}}
\pgfcircle[fill]{\pgfxy(130.00,70.00)}{0.80mm}
\pgfcircle[stroke]{\pgfxy(130.00,70.00)}{0.80mm}
\pgfputat{\pgfxy(118.00,48.00)}{\pgfbox[bottom,left]{\fontsize{9.10}{10.93}\selectfont \makebox[0pt][r]{$j$}}}
\pgfcircle[fill]{\pgfxy(130.00,60.00)}{0.80mm}
\pgfcircle[stroke]{\pgfxy(130.00,60.00)}{0.80mm}
\pgfputat{\pgfxy(134.00,62.00)}{\pgfbox[bottom,left]{\fontsize{9.10}{10.93}\selectfont \makebox[0pt]{A}}}
\pgfellipse[stroke]{\pgfxy(133.50,63.49)}{\pgfxy(6.50,0.00)}{\pgfxy(0.00,4.51)}
\pgfellipse[stroke]{\pgfxy(133.50,53.49)}{\pgfxy(6.50,0.00)}{\pgfxy(0.00,4.51)}
\pgfsetlinewidth{1.20mm}\pgfmoveto{\pgfxy(95.00,65.00)}\pgflineto{\pgfxy(105.00,65.00)}\pgfstroke
\pgfmoveto{\pgfxy(105.00,65.00)}\pgflineto{\pgfxy(102.20,65.70)}\pgflineto{\pgfxy(102.20,64.30)}\pgflineto{\pgfxy(105.00,65.00)}\pgfclosepath\pgffill
\pgfmoveto{\pgfxy(105.00,65.00)}\pgflineto{\pgfxy(102.20,65.70)}\pgflineto{\pgfxy(102.20,64.30)}\pgflineto{\pgfxy(105.00,65.00)}\pgfclosepath\pgfstroke
\pgfputat{\pgfxy(80.00,80.00)}{\pgfbox[bottom,left]{\fontsize{9.10}{10.93}\selectfont \makebox[0pt]{$T'$}}}
\pgfputat{\pgfxy(57.00,59.00)}{\pgfbox[bottom,left]{\fontsize{9.10}{10.93}\selectfont \makebox[0pt][r]{$k-1$}}}
\pgfsetdash{{0.30mm}{0.50mm}}{0mm}\pgfsetlinewidth{0.30mm}\pgfmoveto{\pgfxy(60.00,50.00)}\pgflineto{\pgfxy(60.00,40.00)}\pgfstroke
\pgfcircle[fill]{\pgfxy(70.00,60.00)}{0.80mm}
\pgfsetdash{}{0mm}\pgfcircle[stroke]{\pgfxy(70.00,60.00)}{0.80mm}
\pgfcircle[fill]{\pgfxy(60.00,50.00)}{0.80mm}
\pgfcircle[stroke]{\pgfxy(60.00,50.00)}{0.80mm}
\pgfcircle[fill]{\pgfxy(60.00,40.00)}{0.80mm}
\pgfcircle[stroke]{\pgfxy(60.00,40.00)}{0.80mm}
\pgfmoveto{\pgfxy(70.00,60.00)}\pgflineto{\pgfxy(80.00,70.00)}\pgfstroke
\pgfmoveto{\pgfxy(60.00,40.00)}\pgflineto{\pgfxy(70.00,50.00)}\pgfstroke
\pgfcircle[fill]{\pgfxy(70.00,50.00)}{0.80mm}
\pgfcircle[stroke]{\pgfxy(70.00,50.00)}{0.80mm}
\pgfputat{\pgfxy(73.00,59.00)}{\pgfbox[bottom,left]{\fontsize{9.10}{10.93}\selectfont $k$}}
\pgfcircle[fill]{\pgfxy(60.00,60.00)}{0.80mm}
\pgfcircle[stroke]{\pgfxy(60.00,60.00)}{0.80mm}
\pgfmoveto{\pgfxy(60.00,50.00)}\pgflineto{\pgfxy(60.00,60.00)}\pgfstroke
\pgfputat{\pgfxy(57.00,49.00)}{\pgfbox[bottom,left]{\fontsize{9.10}{10.93}\selectfont \makebox[0pt][r]{$j$}}}
\pgfmoveto{\pgfxy(70.00,60.00)}\pgflineto{\pgfxy(60.00,50.00)}\pgfstroke
\pgfmoveto{\pgfxy(70.00,70.00)}\pgflineto{\pgfxy(70.00,60.00)}\pgfstroke
\pgfellipse[stroke]{\pgfxy(70.00,77.49)}{\pgfxy(5.00,0.00)}{\pgfxy(0.00,7.51)}
\pgfputat{\pgfxy(70.00,77.00)}{\pgfbox[bottom,left]{\fontsize{9.10}{10.93}\selectfont \makebox[0pt]{$A$}}}
\pgfcircle[fill]{\pgfxy(70.00,70.00)}{0.80mm}
\pgfcircle[stroke]{\pgfxy(70.00,70.00)}{0.80mm}
\pgfcircle[fill]{\pgfxy(80.00,70.00)}{0.80mm}
\pgfcircle[stroke]{\pgfxy(80.00,70.00)}{0.80mm}
\pgfputat{\pgfxy(84.00,72.00)}{\pgfbox[bottom,left]{\fontsize{9.10}{10.93}\selectfont \makebox[0pt]{$B$}}}
\pgfellipse[stroke]{\pgfxy(83.50,73.49)}{\pgfxy(6.50,0.00)}{\pgfxy(0.00,4.51)}
\pgfellipse[stroke]{\pgfxy(73.50,53.49)}{\pgfxy(6.50,0.00)}{\pgfxy(0.00,4.51)}
\end{pgfpicture}%
$$
    \item If $k$ is a not sibling of $k-1$ in $T'$, then we get the tree $T = \varphi(\pi) \in \Tnk$ exchanging the labels $k-1$ and $k$ in $T'$. 
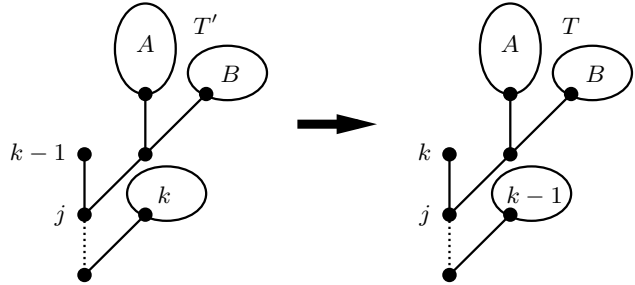
\label{case:b2}
    $$
    \centering
\begin{pgfpicture}{38.00mm}{29.20mm}{122.00mm}{70.00mm}
\pgfsetxvec{\pgfpoint{0.80mm}{0mm}}
\pgfsetyvec{\pgfpoint{0mm}{0.80mm}}
\color[rgb]{0,0,0}\pgfsetlinewidth{0.30mm}\pgfsetdash{}{0mm}
\pgfputat{\pgfxy(140.00,80.00)}{\pgfbox[bottom,left]{\fontsize{9.10}{10.93}\selectfont \makebox[0pt]{$T$}}}
\pgfsetlinewidth{1.20mm}\pgfmoveto{\pgfxy(95.00,65.00)}\pgflineto{\pgfxy(105.00,65.00)}\pgfstroke
\pgfmoveto{\pgfxy(105.00,65.00)}\pgflineto{\pgfxy(102.20,65.70)}\pgflineto{\pgfxy(102.20,64.30)}\pgflineto{\pgfxy(105.00,65.00)}\pgfclosepath\pgffill
\pgfmoveto{\pgfxy(105.00,65.00)}\pgflineto{\pgfxy(102.20,65.70)}\pgflineto{\pgfxy(102.20,64.30)}\pgflineto{\pgfxy(105.00,65.00)}\pgfclosepath\pgfstroke
\pgfputat{\pgfxy(80.00,80.00)}{\pgfbox[bottom,left]{\fontsize{9.10}{10.93}\selectfont \makebox[0pt]{$T'$}}}
\pgfputat{\pgfxy(57.00,59.00)}{\pgfbox[bottom,left]{\fontsize{9.10}{10.93}\selectfont \makebox[0pt][r]{$k-1$}}}
\pgfsetdash{{0.30mm}{0.50mm}}{0mm}\pgfsetlinewidth{0.30mm}\pgfmoveto{\pgfxy(60.00,50.00)}\pgflineto{\pgfxy(60.00,40.00)}\pgfstroke
\pgfcircle[fill]{\pgfxy(70.00,60.00)}{0.80mm}
\pgfsetdash{}{0mm}\pgfcircle[stroke]{\pgfxy(70.00,60.00)}{0.80mm}
\pgfcircle[fill]{\pgfxy(60.00,50.00)}{0.80mm}
\pgfcircle[stroke]{\pgfxy(60.00,50.00)}{0.80mm}
\pgfcircle[fill]{\pgfxy(60.00,40.00)}{0.80mm}
\pgfcircle[stroke]{\pgfxy(60.00,40.00)}{0.80mm}
\pgfmoveto{\pgfxy(70.00,60.00)}\pgflineto{\pgfxy(80.00,70.00)}\pgfstroke
\pgfmoveto{\pgfxy(60.00,40.00)}\pgflineto{\pgfxy(70.00,50.00)}\pgfstroke
\pgfcircle[fill]{\pgfxy(70.00,50.00)}{0.80mm}
\pgfcircle[stroke]{\pgfxy(70.00,50.00)}{0.80mm}
\pgfputat{\pgfxy(72.00,52.00)}{\pgfbox[bottom,left]{\fontsize{9.10}{10.93}\selectfont $k$}}
\pgfcircle[fill]{\pgfxy(60.00,60.00)}{0.80mm}
\pgfcircle[stroke]{\pgfxy(60.00,60.00)}{0.80mm}
\pgfmoveto{\pgfxy(60.00,50.00)}\pgflineto{\pgfxy(60.00,60.00)}\pgfstroke
\pgfputat{\pgfxy(57.00,49.00)}{\pgfbox[bottom,left]{\fontsize{9.10}{10.93}\selectfont \makebox[0pt][r]{$j$}}}
\pgfmoveto{\pgfxy(70.00,60.00)}\pgflineto{\pgfxy(60.00,50.00)}\pgfstroke
\pgfmoveto{\pgfxy(70.00,70.00)}\pgflineto{\pgfxy(70.00,60.00)}\pgfstroke
\pgfellipse[stroke]{\pgfxy(70.00,77.49)}{\pgfxy(5.00,0.00)}{\pgfxy(0.00,7.51)}
\pgfputat{\pgfxy(70.00,77.00)}{\pgfbox[bottom,left]{\fontsize{9.10}{10.93}\selectfont \makebox[0pt]{$A$}}}
\pgfcircle[fill]{\pgfxy(70.00,70.00)}{0.80mm}
\pgfcircle[stroke]{\pgfxy(70.00,70.00)}{0.80mm}
\pgfcircle[fill]{\pgfxy(80.00,70.00)}{0.80mm}
\pgfcircle[stroke]{\pgfxy(80.00,70.00)}{0.80mm}
\pgfputat{\pgfxy(84.00,72.00)}{\pgfbox[bottom,left]{\fontsize{9.10}{10.93}\selectfont \makebox[0pt]{$B$}}}
\pgfellipse[stroke]{\pgfxy(83.50,73.49)}{\pgfxy(6.50,0.00)}{\pgfxy(0.00,4.51)}
\pgfellipse[stroke]{\pgfxy(73.50,53.49)}{\pgfxy(6.50,0.00)}{\pgfxy(0.00,4.51)}
\pgfputat{\pgfxy(117.00,59.00)}{\pgfbox[bottom,left]{\fontsize{9.10}{10.93}\selectfont \makebox[0pt][r]{$k$}}}
\pgfsetdash{{0.30mm}{0.50mm}}{0mm}\pgfmoveto{\pgfxy(120.00,50.00)}\pgflineto{\pgfxy(120.00,40.00)}\pgfstroke
\pgfcircle[fill]{\pgfxy(130.00,60.00)}{0.80mm}
\pgfsetdash{}{0mm}\pgfcircle[stroke]{\pgfxy(130.00,60.00)}{0.80mm}
\pgfcircle[fill]{\pgfxy(120.00,50.00)}{0.80mm}
\pgfcircle[stroke]{\pgfxy(120.00,50.00)}{0.80mm}
\pgfcircle[fill]{\pgfxy(120.00,40.00)}{0.80mm}
\pgfcircle[stroke]{\pgfxy(120.00,40.00)}{0.80mm}
\pgfmoveto{\pgfxy(130.00,60.00)}\pgflineto{\pgfxy(140.00,70.00)}\pgfstroke
\pgfmoveto{\pgfxy(120.00,40.00)}\pgflineto{\pgfxy(130.00,50.00)}\pgfstroke
\pgfcircle[fill]{\pgfxy(130.00,50.00)}{0.80mm}
\pgfcircle[stroke]{\pgfxy(130.00,50.00)}{0.80mm}
\pgfputat{\pgfxy(134.00,52.00)}{\pgfbox[bottom,left]{\fontsize{9.10}{10.93}\selectfont \makebox[0pt]{$k-1$}}}
\pgfcircle[fill]{\pgfxy(120.00,60.00)}{0.80mm}
\pgfcircle[stroke]{\pgfxy(120.00,60.00)}{0.80mm}
\pgfmoveto{\pgfxy(120.00,50.00)}\pgflineto{\pgfxy(120.00,60.00)}\pgfstroke
\pgfputat{\pgfxy(117.00,49.00)}{\pgfbox[bottom,left]{\fontsize{9.10}{10.93}\selectfont \makebox[0pt][r]{$j$}}}
\pgfmoveto{\pgfxy(130.00,60.00)}\pgflineto{\pgfxy(120.00,50.00)}\pgfstroke
\pgfmoveto{\pgfxy(130.00,70.00)}\pgflineto{\pgfxy(130.00,60.00)}\pgfstroke
\pgfellipse[stroke]{\pgfxy(130.00,77.49)}{\pgfxy(5.00,0.00)}{\pgfxy(0.00,7.51)}
\pgfcircle[fill]{\pgfxy(130.00,70.00)}{0.80mm}
\pgfcircle[stroke]{\pgfxy(130.00,70.00)}{0.80mm}
\pgfcircle[fill]{\pgfxy(140.00,70.00)}{0.80mm}
\pgfcircle[stroke]{\pgfxy(140.00,70.00)}{0.80mm}
\pgfellipse[stroke]{\pgfxy(143.50,73.49)}{\pgfxy(6.50,0.00)}{\pgfxy(0.00,4.51)}
\pgfellipse[stroke]{\pgfxy(133.50,53.49)}{\pgfxy(6.50,0.00)}{\pgfxy(0.00,4.51)}
\pgfputat{\pgfxy(130.00,77.00)}{\pgfbox[bottom,left]{\fontsize{9.10}{10.93}\selectfont \makebox[0pt]{$A$}}}
\pgfputat{\pgfxy(144.00,72.00)}{\pgfbox[bottom,left]{\fontsize{9.10}{10.93}\selectfont \makebox[0pt]{$B$}}}
\end{pgfpicture}%
$$
    \end{enumerate}
\end{enumerate}

\begin{figure}[t]
\begin{tabular}{c||c|c|c|c|c|c}
$\pi$ & $1$ & $213$ & $312$ & $21534$ & $31524$ & $41523$
\\
\hline
$\pi_1$ & 1 & 2 & 3 & 2 & 3 & 4
\\
\hline
$\inv(\pi)$ & 0 & 1 & 2 & 3 & 4 & 5
\\
\hline
$\312(\pi)$ & 0 & 0 & 1 & 1 & 2 & 3
\\
\hline
$\varphi(\pi)$ &
$
\centering
\begin{pgfpicture}{1.50mm}{3.70mm}{9.70mm}{10.50mm}
\pgfsetxvec{\pgfpoint{0.70mm}{0mm}}
\pgfsetyvec{\pgfpoint{0mm}{0.70mm}}
\color[rgb]{0,0,0}\pgfsetlinewidth{0.30mm}\pgfsetdash{}{0mm}
\pgfcircle[fill]{\pgfxy(10.00,10.00)}{0.70mm}
\pgfcircle[stroke]{\pgfxy(10.00,10.00)}{0.70mm}
\pgfputat{\pgfxy(7.00,9.00)}{\pgfbox[bottom,left]{\fontsize{7.97}{9.56}\selectfont \makebox[0pt][r]{$1$}}}
\end{pgfpicture}%
$ &
$
\centering
\begin{pgfpicture}{-2.00mm}{3.70mm}{12.50mm}{14.00mm}
\pgfsetxvec{\pgfpoint{0.70mm}{0mm}}
\pgfsetyvec{\pgfpoint{0mm}{0.70mm}}
\color[rgb]{0,0,0}\pgfsetlinewidth{0.30mm}\pgfsetdash{}{0mm}
\pgfcircle[fill]{\pgfxy(5.00,10.00)}{0.70mm}
\pgfcircle[stroke]{\pgfxy(5.00,10.00)}{0.70mm}
\pgfputat{\pgfxy(2.00,9.00)}{\pgfbox[bottom,left]{\fontsize{7.97}{9.56}\selectfont \makebox[0pt][r]{$1$}}}
\pgfmoveto{\pgfxy(5.00,15.00)}\pgflineto{\pgfxy(5.00,10.00)}\pgfstroke
\pgfcircle[fill]{\pgfxy(5.00,15.00)}{0.70mm}
\pgfcircle[stroke]{\pgfxy(5.00,15.00)}{0.70mm}
\pgfcircle[fill]{\pgfxy(10.00,15.00)}{0.70mm}
\pgfcircle[stroke]{\pgfxy(10.00,15.00)}{0.70mm}
\pgfmoveto{\pgfxy(10.00,15.00)}\pgflineto{\pgfxy(5.00,10.00)}\pgfstroke
\pgfputat{\pgfxy(13.00,14.00)}{\pgfbox[bottom,left]{\fontsize{7.97}{9.56}\selectfont $3$}}
\pgfputat{\pgfxy(2.00,14.00)}{\pgfbox[bottom,left]{\fontsize{7.97}{9.56}\selectfont \makebox[0pt][r]{$2$}}}
\end{pgfpicture}%
$ &
$
\centering
\begin{pgfpicture}{4.30mm}{3.70mm}{12.50mm}{17.50mm}
\pgfsetxvec{\pgfpoint{0.70mm}{0mm}}
\pgfsetyvec{\pgfpoint{0mm}{0.70mm}}
\color[rgb]{0,0,0}\pgfsetlinewidth{0.30mm}\pgfsetdash{}{0mm}
\pgfcircle[fill]{\pgfxy(10.00,10.00)}{0.70mm}
\pgfcircle[stroke]{\pgfxy(10.00,10.00)}{0.70mm}
\pgfputat{\pgfxy(13.00,9.00)}{\pgfbox[bottom,left]{\fontsize{7.97}{9.56}\selectfont $1$}}
\pgfmoveto{\pgfxy(10.00,20.00)}\pgflineto{\pgfxy(10.00,15.00)}\pgfstroke
\pgfcircle[fill]{\pgfxy(10.00,20.00)}{0.70mm}
\pgfcircle[stroke]{\pgfxy(10.00,20.00)}{0.70mm}
\pgfcircle[fill]{\pgfxy(10.00,15.00)}{0.70mm}
\pgfcircle[stroke]{\pgfxy(10.00,15.00)}{0.70mm}
\pgfmoveto{\pgfxy(10.00,15.00)}\pgflineto{\pgfxy(10.00,10.00)}\pgfstroke
\pgfputat{\pgfxy(13.00,14.00)}{\pgfbox[bottom,left]{\fontsize{7.97}{9.56}\selectfont $2$}}
\pgfputat{\pgfxy(13.00,19.00)}{\pgfbox[bottom,left]{\fontsize{7.97}{9.56}\selectfont $3$}}
\end{pgfpicture}%
$ &
$
\centering
\begin{pgfpicture}{-2.00mm}{0.20mm}{12.50mm}{17.50mm}
\pgfsetxvec{\pgfpoint{0.70mm}{0mm}}
\pgfsetyvec{\pgfpoint{0mm}{0.70mm}}
\color[rgb]{0,0,0}\pgfsetlinewidth{0.30mm}\pgfsetdash{}{0mm}
\pgfcircle[fill]{\pgfxy(10.00,10.00)}{0.70mm}
\pgfcircle[stroke]{\pgfxy(10.00,10.00)}{0.70mm}
\pgfputat{\pgfxy(2.00,4.00)}{\pgfbox[bottom,left]{\fontsize{7.97}{9.56}\selectfont \makebox[0pt][r]{$1$}}}
\pgfmoveto{\pgfxy(10.00,20.00)}\pgflineto{\pgfxy(10.00,15.00)}\pgfstroke
\pgfcircle[fill]{\pgfxy(10.00,20.00)}{0.70mm}
\pgfcircle[stroke]{\pgfxy(10.00,20.00)}{0.70mm}
\pgfcircle[fill]{\pgfxy(10.00,15.00)}{0.70mm}
\pgfcircle[stroke]{\pgfxy(10.00,15.00)}{0.70mm}
\pgfmoveto{\pgfxy(10.00,15.00)}\pgflineto{\pgfxy(10.00,10.00)}\pgfstroke
\pgfputat{\pgfxy(13.00,19.00)}{\pgfbox[bottom,left]{\fontsize{7.97}{9.56}\selectfont $5$}}
\pgfputat{\pgfxy(13.00,9.00)}{\pgfbox[bottom,left]{\fontsize{7.97}{9.56}\selectfont $3$}}
\pgfcircle[fill]{\pgfxy(5.00,5.00)}{0.70mm}
\pgfcircle[stroke]{\pgfxy(5.00,5.00)}{0.70mm}
\pgfcircle[fill]{\pgfxy(5.00,10.00)}{0.70mm}
\pgfcircle[stroke]{\pgfxy(5.00,10.00)}{0.70mm}
\pgfmoveto{\pgfxy(5.00,5.00)}\pgflineto{\pgfxy(5.00,10.00)}\pgfstroke
\pgfmoveto{\pgfxy(5.00,5.00)}\pgflineto{\pgfxy(10.00,10.00)}\pgfstroke
\pgfputat{\pgfxy(2.00,9.00)}{\pgfbox[bottom,left]{\fontsize{7.97}{9.56}\selectfont \makebox[0pt][r]{$2$}}}
\pgfputat{\pgfxy(13.00,14.00)}{\pgfbox[bottom,left]{\fontsize{7.97}{9.56}\selectfont $4$}}
\end{pgfpicture}%
$ &
$
\centering
\begin{pgfpicture}{-2.00mm}{0.20mm}{12.50mm}{14.00mm}
\pgfsetxvec{\pgfpoint{0.70mm}{0mm}}
\pgfsetyvec{\pgfpoint{0mm}{0.70mm}}
\color[rgb]{0,0,0}\pgfsetlinewidth{0.30mm}\pgfsetdash{}{0mm}
\pgfcircle[fill]{\pgfxy(5.00,15.00)}{0.70mm}
\pgfcircle[stroke]{\pgfxy(5.00,15.00)}{0.70mm}
\pgfputat{\pgfxy(2.00,4.00)}{\pgfbox[bottom,left]{\fontsize{7.97}{9.56}\selectfont \makebox[0pt][r]{$1$}}}
\pgfmoveto{\pgfxy(10.00,15.00)}\pgflineto{\pgfxy(10.00,10.00)}\pgfstroke
\pgfcircle[fill]{\pgfxy(10.00,15.00)}{0.70mm}
\pgfcircle[stroke]{\pgfxy(10.00,15.00)}{0.70mm}
\pgfcircle[fill]{\pgfxy(10.00,10.00)}{0.70mm}
\pgfcircle[stroke]{\pgfxy(10.00,10.00)}{0.70mm}
\pgfmoveto{\pgfxy(5.00,15.00)}\pgflineto{\pgfxy(5.00,9.50)}\pgfstroke
\pgfputat{\pgfxy(13.00,14.00)}{\pgfbox[bottom,left]{\fontsize{7.97}{9.56}\selectfont $5$}}
\pgfputat{\pgfxy(2.00,14.00)}{\pgfbox[bottom,left]{\fontsize{7.97}{9.56}\selectfont \makebox[0pt][r]{$3$}}}
\pgfcircle[fill]{\pgfxy(5.00,5.00)}{0.70mm}
\pgfcircle[stroke]{\pgfxy(5.00,5.00)}{0.70mm}
\pgfcircle[fill]{\pgfxy(5.00,10.00)}{0.70mm}
\pgfcircle[stroke]{\pgfxy(5.00,10.00)}{0.70mm}
\pgfmoveto{\pgfxy(5.00,5.00)}\pgflineto{\pgfxy(5.00,10.00)}\pgfstroke
\pgfmoveto{\pgfxy(5.00,5.00)}\pgflineto{\pgfxy(10.00,10.00)}\pgfstroke
\pgfputat{\pgfxy(2.00,9.00)}{\pgfbox[bottom,left]{\fontsize{7.97}{9.56}\selectfont \makebox[0pt][r]{$2$}}}
\pgfputat{\pgfxy(13.00,9.00)}{\pgfbox[bottom,left]{\fontsize{7.97}{9.56}\selectfont $4$}}
\end{pgfpicture}%
$ &
$
\centering
\begin{pgfpicture}{-2.00mm}{0.20mm}{12.50mm}{14.00mm}
\pgfsetxvec{\pgfpoint{0.70mm}{0mm}}
\pgfsetyvec{\pgfpoint{0mm}{0.70mm}}
\color[rgb]{0,0,0}\pgfsetlinewidth{0.30mm}\pgfsetdash{}{0mm}
\pgfcircle[fill]{\pgfxy(5.00,15.00)}{0.70mm}
\pgfcircle[stroke]{\pgfxy(5.00,15.00)}{0.70mm}
\pgfputat{\pgfxy(2.00,4.00)}{\pgfbox[bottom,left]{\fontsize{7.97}{9.56}\selectfont \makebox[0pt][r]{$1$}}}
\pgfmoveto{\pgfxy(10.00,15.00)}\pgflineto{\pgfxy(10.00,10.00)}\pgfstroke
\pgfcircle[fill]{\pgfxy(10.00,15.00)}{0.70mm}
\pgfcircle[stroke]{\pgfxy(10.00,15.00)}{0.70mm}
\pgfcircle[fill]{\pgfxy(10.00,10.00)}{0.70mm}
\pgfcircle[stroke]{\pgfxy(10.00,10.00)}{0.70mm}
\pgfmoveto{\pgfxy(5.00,15.00)}\pgflineto{\pgfxy(5.00,9.50)}\pgfstroke
\pgfputat{\pgfxy(13.00,14.00)}{\pgfbox[bottom,left]{\fontsize{7.97}{9.56}\selectfont $5$}}
\pgfputat{\pgfxy(2.00,14.00)}{\pgfbox[bottom,left]{\fontsize{7.97}{9.56}\selectfont \makebox[0pt][r]{$4$}}}
\pgfcircle[fill]{\pgfxy(5.00,5.00)}{0.70mm}
\pgfcircle[stroke]{\pgfxy(5.00,5.00)}{0.70mm}
\pgfcircle[fill]{\pgfxy(5.00,10.00)}{0.70mm}
\pgfcircle[stroke]{\pgfxy(5.00,10.00)}{0.70mm}
\pgfmoveto{\pgfxy(5.00,5.00)}\pgflineto{\pgfxy(5.00,10.00)}\pgfstroke
\pgfmoveto{\pgfxy(5.00,5.00)}\pgflineto{\pgfxy(10.00,10.00)}\pgfstroke
\pgfputat{\pgfxy(2.00,9.00)}{\pgfbox[bottom,left]{\fontsize{7.97}{9.56}\selectfont \makebox[0pt][r]{$2$}}}
\pgfputat{\pgfxy(13.00,9.00)}{\pgfbox[bottom,left]{\fontsize{7.97}{9.56}\selectfont $3$}}
\end{pgfpicture}%
$
\\
\hline
$p(\varphi(\pi))$ & 1 & 2 & 3 & 2 & 3 & 4
\\
\hline
\hline
$\pi$ & $51423$ & $5471623$ & $6471523$ & $548691723$ & $648591723$ & 748591623
\\
\hline
$\pi_1$ & 5 & 5 & 6 & 5 & 6 & 7
\\
\hline
$\inv(\pi)$ & 6 & 13 & 14 & 21 & 22 & 23
\\
\hline
$\312(\pi)$ & 4 & 4 & 5 & 5 & 6 & 7
\\
\hline
$\varphi(\pi)$ &
$
\centering
\begin{pgfpicture}{-2.00mm}{0.20mm}{12.50mm}{14.00mm}
\pgfsetxvec{\pgfpoint{0.70mm}{0mm}}
\pgfsetyvec{\pgfpoint{0mm}{0.70mm}}
\color[rgb]{0,0,0}\pgfsetlinewidth{0.30mm}\pgfsetdash{}{0mm}
\pgfcircle[fill]{\pgfxy(5.00,15.00)}{0.70mm}
\pgfcircle[stroke]{\pgfxy(5.00,15.00)}{0.70mm}
\pgfputat{\pgfxy(2.00,4.00)}{\pgfbox[bottom,left]{\fontsize{7.97}{9.56}\selectfont \makebox[0pt][r]{$1$}}}
\pgfmoveto{\pgfxy(10.00,15.00)}\pgflineto{\pgfxy(10.00,10.00)}\pgfstroke
\pgfcircle[fill]{\pgfxy(10.00,15.00)}{0.70mm}
\pgfcircle[stroke]{\pgfxy(10.00,15.00)}{0.70mm}
\pgfcircle[fill]{\pgfxy(10.00,10.00)}{0.70mm}
\pgfcircle[stroke]{\pgfxy(10.00,10.00)}{0.70mm}
\pgfmoveto{\pgfxy(5.00,15.00)}\pgflineto{\pgfxy(5.00,9.50)}\pgfstroke
\pgfputat{\pgfxy(13.00,14.00)}{\pgfbox[bottom,left]{\fontsize{7.97}{9.56}\selectfont $4$}}
\pgfputat{\pgfxy(2.00,14.00)}{\pgfbox[bottom,left]{\fontsize{7.97}{9.56}\selectfont \makebox[0pt][r]{$5$}}}
\pgfcircle[fill]{\pgfxy(5.00,5.00)}{0.70mm}
\pgfcircle[stroke]{\pgfxy(5.00,5.00)}{0.70mm}
\pgfcircle[fill]{\pgfxy(5.00,10.00)}{0.70mm}
\pgfcircle[stroke]{\pgfxy(5.00,10.00)}{0.70mm}
\pgfmoveto{\pgfxy(5.00,5.00)}\pgflineto{\pgfxy(5.00,10.00)}\pgfstroke
\pgfmoveto{\pgfxy(5.00,5.00)}\pgflineto{\pgfxy(10.00,10.00)}\pgfstroke
\pgfputat{\pgfxy(2.00,9.00)}{\pgfbox[bottom,left]{\fontsize{7.97}{9.56}\selectfont \makebox[0pt][r]{$2$}}}
\pgfputat{\pgfxy(13.00,9.00)}{\pgfbox[bottom,left]{\fontsize{7.97}{9.56}\selectfont $3$}}
\end{pgfpicture}%
$ &
$
\centering
\begin{pgfpicture}{-2.00mm}{0.20mm}{12.50mm}{17.50mm}
\pgfsetxvec{\pgfpoint{0.70mm}{0mm}}
\pgfsetyvec{\pgfpoint{0mm}{0.70mm}}
\color[rgb]{0,0,0}\pgfsetlinewidth{0.30mm}\pgfsetdash{}{0mm}
\pgfcircle[fill]{\pgfxy(5.00,15.00)}{0.70mm}
\pgfcircle[stroke]{\pgfxy(5.00,15.00)}{0.70mm}
\pgfputat{\pgfxy(2.00,4.00)}{\pgfbox[bottom,left]{\fontsize{7.97}{9.56}\selectfont \makebox[0pt][r]{$1$}}}
\pgfmoveto{\pgfxy(10.00,15.00)}\pgflineto{\pgfxy(10.00,10.00)}\pgfstroke
\pgfcircle[fill]{\pgfxy(10.00,15.00)}{0.70mm}
\pgfcircle[stroke]{\pgfxy(10.00,15.00)}{0.70mm}
\pgfcircle[fill]{\pgfxy(10.00,10.00)}{0.70mm}
\pgfcircle[stroke]{\pgfxy(10.00,10.00)}{0.70mm}
\pgfmoveto{\pgfxy(5.00,15.00)}\pgflineto{\pgfxy(5.00,9.50)}\pgfstroke
\pgfputat{\pgfxy(2.00,14.00)}{\pgfbox[bottom,left]{\fontsize{7.97}{9.56}\selectfont \makebox[0pt][r]{$4$}}}
\pgfputat{\pgfxy(2.00,19.00)}{\pgfbox[bottom,left]{\fontsize{7.97}{9.56}\selectfont \makebox[0pt][r]{$5$}}}
\pgfcircle[fill]{\pgfxy(5.00,5.00)}{0.70mm}
\pgfcircle[stroke]{\pgfxy(5.00,5.00)}{0.70mm}
\pgfcircle[fill]{\pgfxy(5.00,10.00)}{0.70mm}
\pgfcircle[stroke]{\pgfxy(5.00,10.00)}{0.70mm}
\pgfmoveto{\pgfxy(5.00,5.00)}\pgflineto{\pgfxy(5.00,10.00)}\pgfstroke
\pgfmoveto{\pgfxy(5.00,5.00)}\pgflineto{\pgfxy(10.00,10.00)}\pgfstroke
\pgfputat{\pgfxy(2.00,9.00)}{\pgfbox[bottom,left]{\fontsize{7.97}{9.56}\selectfont \makebox[0pt][r]{$2$}}}
\pgfputat{\pgfxy(13.00,9.00)}{\pgfbox[bottom,left]{\fontsize{7.97}{9.56}\selectfont $3$}}
\pgfcircle[fill]{\pgfxy(10.00,20.00)}{0.70mm}
\pgfcircle[stroke]{\pgfxy(10.00,20.00)}{0.70mm}
\pgfmoveto{\pgfxy(5.00,20.00)}\pgflineto{\pgfxy(5.00,15.00)}\pgfstroke
\pgfmoveto{\pgfxy(5.00,15.00)}\pgflineto{\pgfxy(10.00,20.00)}\pgfstroke
\pgfcircle[fill]{\pgfxy(5.00,20.00)}{0.70mm}
\pgfcircle[stroke]{\pgfxy(5.00,20.00)}{0.70mm}
\pgfputat{\pgfxy(13.00,14.00)}{\pgfbox[bottom,left]{\fontsize{7.97}{9.56}\selectfont $6$}}
\pgfputat{\pgfxy(13.00,19.00)}{\pgfbox[bottom,left]{\fontsize{7.97}{9.56}\selectfont $7$}}
\end{pgfpicture}%
$ &
$
\centering
\begin{pgfpicture}{-2.00mm}{0.20mm}{12.50mm}{17.50mm}
\pgfsetxvec{\pgfpoint{0.70mm}{0mm}}
\pgfsetyvec{\pgfpoint{0mm}{0.70mm}}
\color[rgb]{0,0,0}\pgfsetlinewidth{0.30mm}\pgfsetdash{}{0mm}
\pgfcircle[fill]{\pgfxy(5.00,15.00)}{0.70mm}
\pgfcircle[stroke]{\pgfxy(5.00,15.00)}{0.70mm}
\pgfputat{\pgfxy(2.00,4.00)}{\pgfbox[bottom,left]{\fontsize{7.97}{9.56}\selectfont \makebox[0pt][r]{$1$}}}
\pgfmoveto{\pgfxy(10.00,15.00)}\pgflineto{\pgfxy(10.00,10.00)}\pgfstroke
\pgfcircle[fill]{\pgfxy(10.00,15.00)}{0.70mm}
\pgfcircle[stroke]{\pgfxy(10.00,15.00)}{0.70mm}
\pgfcircle[fill]{\pgfxy(10.00,10.00)}{0.70mm}
\pgfcircle[stroke]{\pgfxy(10.00,10.00)}{0.70mm}
\pgfmoveto{\pgfxy(5.00,15.00)}\pgflineto{\pgfxy(5.00,9.50)}\pgfstroke
\pgfputat{\pgfxy(2.00,14.00)}{\pgfbox[bottom,left]{\fontsize{7.97}{9.56}\selectfont \makebox[0pt][r]{$4$}}}
\pgfputat{\pgfxy(2.00,19.00)}{\pgfbox[bottom,left]{\fontsize{7.97}{9.56}\selectfont \makebox[0pt][r]{$6$}}}
\pgfcircle[fill]{\pgfxy(5.00,5.00)}{0.70mm}
\pgfcircle[stroke]{\pgfxy(5.00,5.00)}{0.70mm}
\pgfcircle[fill]{\pgfxy(5.00,10.00)}{0.70mm}
\pgfcircle[stroke]{\pgfxy(5.00,10.00)}{0.70mm}
\pgfmoveto{\pgfxy(5.00,5.00)}\pgflineto{\pgfxy(5.00,10.00)}\pgfstroke
\pgfmoveto{\pgfxy(5.00,5.00)}\pgflineto{\pgfxy(10.00,10.00)}\pgfstroke
\pgfputat{\pgfxy(2.00,9.00)}{\pgfbox[bottom,left]{\fontsize{7.97}{9.56}\selectfont \makebox[0pt][r]{$2$}}}
\pgfputat{\pgfxy(13.00,9.00)}{\pgfbox[bottom,left]{\fontsize{7.97}{9.56}\selectfont $3$}}
\pgfcircle[fill]{\pgfxy(10.00,20.00)}{0.70mm}
\pgfcircle[stroke]{\pgfxy(10.00,20.00)}{0.70mm}
\pgfmoveto{\pgfxy(5.00,20.00)}\pgflineto{\pgfxy(5.00,15.00)}\pgfstroke
\pgfmoveto{\pgfxy(5.00,15.00)}\pgflineto{\pgfxy(10.00,20.00)}\pgfstroke
\pgfcircle[fill]{\pgfxy(5.00,20.00)}{0.70mm}
\pgfcircle[stroke]{\pgfxy(5.00,20.00)}{0.70mm}
\pgfputat{\pgfxy(13.00,14.00)}{\pgfbox[bottom,left]{\fontsize{7.97}{9.56}\selectfont $5$}}
\pgfputat{\pgfxy(13.00,19.00)}{\pgfbox[bottom,left]{\fontsize{7.97}{9.56}\selectfont $7$}}
\end{pgfpicture}%
$ &
$
\centering
\begin{pgfpicture}{-2.00mm}{0.20mm}{16.00mm}{21.00mm}
\pgfsetxvec{\pgfpoint{0.70mm}{0mm}}
\pgfsetyvec{\pgfpoint{0mm}{0.70mm}}
\color[rgb]{0,0,0}\pgfsetlinewidth{0.30mm}\pgfsetdash{}{0mm}
\pgfputat{\pgfxy(2.00,4.00)}{\pgfbox[bottom,left]{\fontsize{7.97}{9.56}\selectfont \makebox[0pt][r]{$1$}}}
\pgfmoveto{\pgfxy(10.00,15.00)}\pgflineto{\pgfxy(10.00,10.00)}\pgfstroke
\pgfcircle[fill]{\pgfxy(10.00,15.00)}{0.70mm}
\pgfcircle[stroke]{\pgfxy(10.00,15.00)}{0.70mm}
\pgfcircle[fill]{\pgfxy(10.00,10.00)}{0.70mm}
\pgfcircle[stroke]{\pgfxy(10.00,10.00)}{0.70mm}
\pgfmoveto{\pgfxy(5.00,15.00)}\pgflineto{\pgfxy(5.00,9.50)}\pgfstroke
\pgfcircle[fill]{\pgfxy(5.00,5.00)}{0.70mm}
\pgfcircle[stroke]{\pgfxy(5.00,5.00)}{0.70mm}
\pgfcircle[fill]{\pgfxy(5.00,10.00)}{0.70mm}
\pgfcircle[stroke]{\pgfxy(5.00,10.00)}{0.70mm}
\pgfmoveto{\pgfxy(5.00,5.00)}\pgflineto{\pgfxy(5.00,10.00)}\pgfstroke
\pgfmoveto{\pgfxy(5.00,5.00)}\pgflineto{\pgfxy(10.00,10.00)}\pgfstroke
\pgfputat{\pgfxy(2.00,9.00)}{\pgfbox[bottom,left]{\fontsize{7.97}{9.56}\selectfont \makebox[0pt][r]{$2$}}}
\pgfputat{\pgfxy(13.00,9.00)}{\pgfbox[bottom,left]{\fontsize{7.97}{9.56}\selectfont $3$}}
\pgfputat{\pgfxy(13.00,14.00)}{\pgfbox[bottom,left]{\fontsize{7.97}{9.56}\selectfont $7$}}
\pgfcircle[fill]{\pgfxy(10.00,20.00)}{0.70mm}
\pgfcircle[stroke]{\pgfxy(10.00,20.00)}{0.70mm}
\pgfputat{\pgfxy(13.00,19.00)}{\pgfbox[bottom,left]{\fontsize{7.97}{9.56}\selectfont $6$}}
\pgfputat{\pgfxy(7.00,24.00)}{\pgfbox[bottom,left]{\fontsize{7.97}{9.56}\selectfont \makebox[0pt][r]{$8$}}}
\pgfcircle[fill]{\pgfxy(15.00,25.00)}{0.70mm}
\pgfcircle[stroke]{\pgfxy(15.00,25.00)}{0.70mm}
\pgfmoveto{\pgfxy(10.00,25.00)}\pgflineto{\pgfxy(10.00,20.00)}\pgfstroke
\pgfmoveto{\pgfxy(10.00,20.00)}\pgflineto{\pgfxy(15.00,25.00)}\pgfstroke
\pgfcircle[fill]{\pgfxy(10.00,25.00)}{0.70mm}
\pgfcircle[stroke]{\pgfxy(10.00,25.00)}{0.70mm}
\pgfputat{\pgfxy(18.00,24.00)}{\pgfbox[bottom,left]{\fontsize{7.97}{9.56}\selectfont $9$}}
\pgfcircle[fill]{\pgfxy(5.00,15.00)}{0.70mm}
\pgfcircle[stroke]{\pgfxy(5.00,15.00)}{0.70mm}
\pgfcircle[fill]{\pgfxy(5.00,20.00)}{0.70mm}
\pgfcircle[stroke]{\pgfxy(5.00,20.00)}{0.70mm}
\pgfmoveto{\pgfxy(10.00,20.00)}\pgflineto{\pgfxy(5.00,15.00)}\pgfstroke
\pgfmoveto{\pgfxy(5.00,20.00)}\pgflineto{\pgfxy(5.00,15.00)}\pgfstroke
\pgfputat{\pgfxy(2.00,19.00)}{\pgfbox[bottom,left]{\fontsize{7.97}{9.56}\selectfont \makebox[0pt][r]{$5$}}}
\pgfputat{\pgfxy(2.00,14.00)}{\pgfbox[bottom,left]{\fontsize{7.97}{9.56}\selectfont \makebox[0pt][r]{$4$}}}
\end{pgfpicture}%
$ &
$
\centering
\begin{pgfpicture}{-2.00mm}{0.20mm}{12.50mm}{21.00mm}
\pgfsetxvec{\pgfpoint{0.70mm}{0mm}}
\pgfsetyvec{\pgfpoint{0mm}{0.70mm}}
\color[rgb]{0,0,0}\pgfsetlinewidth{0.30mm}\pgfsetdash{}{0mm}
\pgfputat{\pgfxy(2.00,4.00)}{\pgfbox[bottom,left]{\fontsize{7.97}{9.56}\selectfont \makebox[0pt][r]{$1$}}}
\pgfmoveto{\pgfxy(10.00,15.00)}\pgflineto{\pgfxy(10.00,10.00)}\pgfstroke
\pgfcircle[fill]{\pgfxy(10.00,15.00)}{0.70mm}
\pgfcircle[stroke]{\pgfxy(10.00,15.00)}{0.70mm}
\pgfcircle[fill]{\pgfxy(10.00,10.00)}{0.70mm}
\pgfcircle[stroke]{\pgfxy(10.00,10.00)}{0.70mm}
\pgfmoveto{\pgfxy(5.00,15.00)}\pgflineto{\pgfxy(5.00,9.50)}\pgfstroke
\pgfcircle[fill]{\pgfxy(5.00,5.00)}{0.70mm}
\pgfcircle[stroke]{\pgfxy(5.00,5.00)}{0.70mm}
\pgfcircle[fill]{\pgfxy(5.00,10.00)}{0.70mm}
\pgfcircle[stroke]{\pgfxy(5.00,10.00)}{0.70mm}
\pgfmoveto{\pgfxy(5.00,5.00)}\pgflineto{\pgfxy(5.00,10.00)}\pgfstroke
\pgfmoveto{\pgfxy(5.00,5.00)}\pgflineto{\pgfxy(10.00,10.00)}\pgfstroke
\pgfputat{\pgfxy(2.00,9.00)}{\pgfbox[bottom,left]{\fontsize{7.97}{9.56}\selectfont \makebox[0pt][r]{$2$}}}
\pgfputat{\pgfxy(13.00,9.00)}{\pgfbox[bottom,left]{\fontsize{7.97}{9.56}\selectfont $3$}}
\pgfputat{\pgfxy(13.00,14.00)}{\pgfbox[bottom,left]{\fontsize{7.97}{9.56}\selectfont $7$}}
\pgfcircle[fill]{\pgfxy(10.00,20.00)}{0.70mm}
\pgfcircle[stroke]{\pgfxy(10.00,20.00)}{0.70mm}
\pgfputat{\pgfxy(2.00,24.00)}{\pgfbox[bottom,left]{\fontsize{7.97}{9.56}\selectfont \makebox[0pt][r]{$6$}}}
\pgfputat{\pgfxy(13.00,19.00)}{\pgfbox[bottom,left]{\fontsize{7.97}{9.56}\selectfont $8$}}
\pgfcircle[fill]{\pgfxy(10.00,25.00)}{0.70mm}
\pgfcircle[stroke]{\pgfxy(10.00,25.00)}{0.70mm}
\pgfmoveto{\pgfxy(5.00,25.00)}\pgflineto{\pgfxy(5.00,20.00)}\pgfstroke
\pgfmoveto{\pgfxy(10.00,25.00)}\pgflineto{\pgfxy(5.00,20.00)}\pgfstroke
\pgfcircle[fill]{\pgfxy(5.00,25.00)}{0.70mm}
\pgfcircle[stroke]{\pgfxy(5.00,25.00)}{0.70mm}
\pgfputat{\pgfxy(13.00,24.00)}{\pgfbox[bottom,left]{\fontsize{7.97}{9.56}\selectfont $9$}}
\pgfcircle[fill]{\pgfxy(5.00,15.00)}{0.70mm}
\pgfcircle[stroke]{\pgfxy(5.00,15.00)}{0.70mm}
\pgfcircle[fill]{\pgfxy(5.00,20.00)}{0.70mm}
\pgfcircle[stroke]{\pgfxy(5.00,20.00)}{0.70mm}
\pgfmoveto{\pgfxy(10.00,20.00)}\pgflineto{\pgfxy(5.00,15.00)}\pgfstroke
\pgfmoveto{\pgfxy(5.00,20.00)}\pgflineto{\pgfxy(5.00,15.00)}\pgfstroke
\pgfputat{\pgfxy(2.00,19.00)}{\pgfbox[bottom,left]{\fontsize{7.97}{9.56}\selectfont \makebox[0pt][r]{$5$}}}
\pgfputat{\pgfxy(2.00,14.00)}{\pgfbox[bottom,left]{\fontsize{7.97}{9.56}\selectfont \makebox[0pt][r]{$4$}}}
\end{pgfpicture}%
$ &
$
\centering
\begin{pgfpicture}{-2.00mm}{0.20mm}{12.50mm}{21.00mm}
\pgfsetxvec{\pgfpoint{0.70mm}{0mm}}
\pgfsetyvec{\pgfpoint{0mm}{0.70mm}}
\color[rgb]{0,0,0}\pgfsetlinewidth{0.30mm}\pgfsetdash{}{0mm}
\pgfputat{\pgfxy(2.00,4.00)}{\pgfbox[bottom,left]{\fontsize{7.97}{9.56}\selectfont \makebox[0pt][r]{$1$}}}
\pgfmoveto{\pgfxy(10.00,15.00)}\pgflineto{\pgfxy(10.00,10.00)}\pgfstroke
\pgfcircle[fill]{\pgfxy(10.00,15.00)}{0.70mm}
\pgfcircle[stroke]{\pgfxy(10.00,15.00)}{0.70mm}
\pgfcircle[fill]{\pgfxy(10.00,10.00)}{0.70mm}
\pgfcircle[stroke]{\pgfxy(10.00,10.00)}{0.70mm}
\pgfmoveto{\pgfxy(5.00,15.00)}\pgflineto{\pgfxy(5.00,9.50)}\pgfstroke
\pgfcircle[fill]{\pgfxy(5.00,5.00)}{0.70mm}
\pgfcircle[stroke]{\pgfxy(5.00,5.00)}{0.70mm}
\pgfcircle[fill]{\pgfxy(5.00,10.00)}{0.70mm}
\pgfcircle[stroke]{\pgfxy(5.00,10.00)}{0.70mm}
\pgfmoveto{\pgfxy(5.00,5.00)}\pgflineto{\pgfxy(5.00,10.00)}\pgfstroke
\pgfmoveto{\pgfxy(5.00,5.00)}\pgflineto{\pgfxy(10.00,10.00)}\pgfstroke
\pgfputat{\pgfxy(2.00,9.00)}{\pgfbox[bottom,left]{\fontsize{7.97}{9.56}\selectfont \makebox[0pt][r]{$2$}}}
\pgfputat{\pgfxy(13.00,9.00)}{\pgfbox[bottom,left]{\fontsize{7.97}{9.56}\selectfont $3$}}
\pgfputat{\pgfxy(13.00,14.00)}{\pgfbox[bottom,left]{\fontsize{7.97}{9.56}\selectfont $6$}}
\pgfcircle[fill]{\pgfxy(10.00,20.00)}{0.70mm}
\pgfcircle[stroke]{\pgfxy(10.00,20.00)}{0.70mm}
\pgfputat{\pgfxy(2.00,24.00)}{\pgfbox[bottom,left]{\fontsize{7.97}{9.56}\selectfont \makebox[0pt][r]{$7$}}}
\pgfputat{\pgfxy(13.00,19.00)}{\pgfbox[bottom,left]{\fontsize{7.97}{9.56}\selectfont $8$}}
\pgfcircle[fill]{\pgfxy(10.00,25.00)}{0.70mm}
\pgfcircle[stroke]{\pgfxy(10.00,25.00)}{0.70mm}
\pgfmoveto{\pgfxy(5.00,25.00)}\pgflineto{\pgfxy(5.00,20.00)}\pgfstroke
\pgfmoveto{\pgfxy(10.00,25.00)}\pgflineto{\pgfxy(5.00,20.00)}\pgfstroke
\pgfcircle[fill]{\pgfxy(5.00,25.00)}{0.70mm}
\pgfcircle[stroke]{\pgfxy(5.00,25.00)}{0.70mm}
\pgfputat{\pgfxy(13.00,24.00)}{\pgfbox[bottom,left]{\fontsize{7.97}{9.56}\selectfont $9$}}
\pgfcircle[fill]{\pgfxy(5.00,15.00)}{0.70mm}
\pgfcircle[stroke]{\pgfxy(5.00,15.00)}{0.70mm}
\pgfcircle[fill]{\pgfxy(5.00,20.00)}{0.70mm}
\pgfcircle[stroke]{\pgfxy(5.00,20.00)}{0.70mm}
\pgfmoveto{\pgfxy(10.00,20.00)}\pgflineto{\pgfxy(5.00,15.00)}\pgfstroke
\pgfmoveto{\pgfxy(5.00,20.00)}\pgflineto{\pgfxy(5.00,15.00)}\pgfstroke
\pgfputat{\pgfxy(2.00,19.00)}{\pgfbox[bottom,left]{\fontsize{7.97}{9.56}\selectfont \makebox[0pt][r]{$5$}}}
\pgfputat{\pgfxy(2.00,14.00)}{\pgfbox[bottom,left]{\fontsize{7.97}{9.56}\selectfont \makebox[0pt][r]{$4$}}}
\end{pgfpicture}%
$
\\
\hline
$p(\varphi(\pi))$ & 5 & 5 & 6 & 5 & 6 & 7
\\
\end{tabular}
\caption{Constructing a tree $\varphi(748591623)$ by a recursive algorithm}
\end{figure}

\begin{thm}
For all $n\ge 1$ and $k\in[n]$ The mapping $\varphi$ is a bijection between $\Ank$ and $\Tnk$ satisfying $$\pi_1=p(\varphi(\pi)).$$
\end{thm}

\begin{proof}
It is sufficient to construct the inverse mapping of $\varphi$.
Given $T \in \Tnk$ ($k=p(T)$), we define $\pi=\varphi^{-1}(T)$ recursively as follows:
\begin{enumerate}[(A)]
\item If $k-1$ is a parent of $k$ in $T$, then let $m(>k)$ be the another child of $k-1$ ($m=\infty$ if $k-1$ has only child $k$) and $s(>k)$ be a sibling of $k-1$ ($s=\infty$ if $k-1$ has no sibling).
\begin{enumerate}[(1)]
\item If $m<\infty$ and $m<s$ (Case~\ref{case:a}), then define $T'$ by deleting vertex $k-1$ and $k$ and their adjacent edges from $T$ and adding new edge $(m,j)$ where $j$ is a parent of $k-1$ in $T$. 
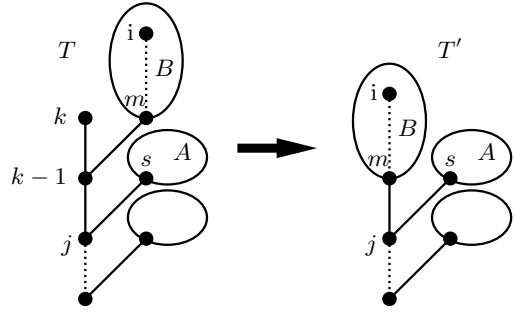
\label{case:A1}
    $$
    \centering
\begin{pgfpicture}{38.00mm}{29.20mm}{122.00mm}{74.00mm}
\pgfsetxvec{\pgfpoint{0.80mm}{0mm}}
\pgfsetyvec{\pgfpoint{0mm}{0.80mm}}
\color[rgb]{0,0,0}\pgfsetlinewidth{0.30mm}\pgfsetdash{}{0mm}
\pgfputat{\pgfxy(67.00,80.00)}{\pgfbox[bottom,left]{\fontsize{9.10}{10.93}\selectfont \makebox[0pt]{$T$}}}
\pgfsetdash{{0.30mm}{0.50mm}}{0mm}\pgfmoveto{\pgfxy(70.00,50.00)}\pgflineto{\pgfxy(70.00,40.00)}\pgfstroke
\pgfcircle[fill]{\pgfxy(70.00,50.00)}{0.80mm}
\pgfsetdash{}{0mm}\pgfcircle[stroke]{\pgfxy(70.00,50.00)}{0.80mm}
\pgfcircle[fill]{\pgfxy(70.00,40.00)}{0.80mm}
\pgfcircle[stroke]{\pgfxy(70.00,40.00)}{0.80mm}
\pgfmoveto{\pgfxy(70.00,50.00)}\pgflineto{\pgfxy(80.00,60.00)}\pgfstroke
\pgfcircle[fill]{\pgfxy(70.00,70.00)}{0.80mm}
\pgfcircle[stroke]{\pgfxy(70.00,70.00)}{0.80mm}
\pgfmoveto{\pgfxy(70.00,40.00)}\pgflineto{\pgfxy(80.00,50.00)}\pgfstroke
\pgfcircle[fill]{\pgfxy(80.00,50.00)}{0.80mm}
\pgfcircle[stroke]{\pgfxy(80.00,50.00)}{0.80mm}
\pgfputat{\pgfxy(67.00,69.00)}{\pgfbox[bottom,left]{\fontsize{9.10}{10.93}\selectfont \makebox[0pt][r]{$k$}}}
\pgfputat{\pgfxy(67.00,59.00)}{\pgfbox[bottom,left]{\fontsize{9.10}{10.93}\selectfont \makebox[0pt][r]{$k-1$}}}
\pgfcircle[fill]{\pgfxy(70.00,60.00)}{0.80mm}
\pgfcircle[stroke]{\pgfxy(70.00,60.00)}{0.80mm}
\pgfmoveto{\pgfxy(70.00,50.00)}\pgflineto{\pgfxy(70.00,60.00)}\pgfstroke
\pgfmoveto{\pgfxy(80.00,70.00)}\pgflineto{\pgfxy(70.00,60.00)}\pgfstroke
\pgfmoveto{\pgfxy(70.00,70.00)}\pgflineto{\pgfxy(70.00,60.00)}\pgfstroke
\pgfellipse[stroke]{\pgfxy(80.00,79.49)}{\pgfxy(6.00,0.00)}{\pgfxy(0.00,9.51)}
\pgfputat{\pgfxy(83.00,77.00)}{\pgfbox[bottom,left]{\fontsize{9.10}{10.93}\selectfont \makebox[0pt]{$B$}}}
\pgfcircle[fill]{\pgfxy(80.00,70.00)}{0.80mm}
\pgfcircle[stroke]{\pgfxy(80.00,70.00)}{0.80mm}
\pgfputat{\pgfxy(80.00,72.00)}{\pgfbox[bottom,left]{\fontsize{9.10}{10.93}\selectfont \makebox[0pt][r]{$m$}}}
\pgfputat{\pgfxy(68.00,48.00)}{\pgfbox[bottom,left]{\fontsize{9.10}{10.93}\selectfont \makebox[0pt][r]{$j$}}}
\pgfcircle[fill]{\pgfxy(80.00,60.00)}{0.80mm}
\pgfcircle[stroke]{\pgfxy(80.00,60.00)}{0.80mm}
\pgfputat{\pgfxy(86.00,63.00)}{\pgfbox[bottom,left]{\fontsize{9.10}{10.93}\selectfont \makebox[0pt]{$A$}}}
\pgfputat{\pgfxy(80.00,62.00)}{\pgfbox[bottom,left]{\fontsize{9.10}{10.93}\selectfont \makebox[0pt]{$s$}}}
\pgfellipse[stroke]{\pgfxy(83.50,63.49)}{\pgfxy(6.50,0.00)}{\pgfxy(0.00,4.51)}
\pgfellipse[stroke]{\pgfxy(83.50,53.49)}{\pgfxy(6.50,0.00)}{\pgfxy(0.00,4.51)}
\pgfsetlinewidth{1.20mm}\pgfmoveto{\pgfxy(95.00,65.00)}\pgflineto{\pgfxy(105.00,65.00)}\pgfstroke
\pgfmoveto{\pgfxy(105.00,65.00)}\pgflineto{\pgfxy(102.20,65.70)}\pgflineto{\pgfxy(102.20,64.30)}\pgflineto{\pgfxy(105.00,65.00)}\pgfclosepath\pgffill
\pgfmoveto{\pgfxy(105.00,65.00)}\pgflineto{\pgfxy(102.20,65.70)}\pgflineto{\pgfxy(102.20,64.30)}\pgflineto{\pgfxy(105.00,65.00)}\pgfclosepath\pgfstroke
\pgfputat{\pgfxy(130.00,80.00)}{\pgfbox[bottom,left]{\fontsize{9.10}{10.93}\selectfont \makebox[0pt]{$T'$}}}
\pgfcircle[fill]{\pgfxy(80.00,84.00)}{0.80mm}
\pgfsetlinewidth{0.30mm}\pgfcircle[stroke]{\pgfxy(80.00,84.00)}{0.80mm}
\pgfsetdash{{0.30mm}{0.50mm}}{0mm}\pgfmoveto{\pgfxy(80.00,84.00)}\pgflineto{\pgfxy(80.00,70.00)}\pgfstroke
\pgfputat{\pgfxy(78.00,83.00)}{\pgfbox[bottom,left]{\fontsize{9.10}{10.93}\selectfont \makebox[0pt][r]{i}}}
\pgfmoveto{\pgfxy(120.00,50.00)}\pgflineto{\pgfxy(120.00,40.00)}\pgfstroke
\pgfcircle[fill]{\pgfxy(120.00,50.00)}{0.80mm}
\pgfsetdash{}{0mm}\pgfcircle[stroke]{\pgfxy(120.00,50.00)}{0.80mm}
\pgfcircle[fill]{\pgfxy(120.00,40.00)}{0.80mm}
\pgfcircle[stroke]{\pgfxy(120.00,40.00)}{0.80mm}
\pgfmoveto{\pgfxy(120.00,50.00)}\pgflineto{\pgfxy(130.00,60.00)}\pgfstroke
\pgfmoveto{\pgfxy(120.00,40.00)}\pgflineto{\pgfxy(130.00,50.00)}\pgfstroke
\pgfcircle[fill]{\pgfxy(130.00,50.00)}{0.80mm}
\pgfcircle[stroke]{\pgfxy(130.00,50.00)}{0.80mm}
\pgfmoveto{\pgfxy(120.00,50.00)}\pgflineto{\pgfxy(120.00,60.00)}\pgfstroke
\pgfellipse[stroke]{\pgfxy(120.00,69.49)}{\pgfxy(6.00,0.00)}{\pgfxy(0.00,9.51)}
\pgfputat{\pgfxy(123.00,67.00)}{\pgfbox[bottom,left]{\fontsize{9.10}{10.93}\selectfont \makebox[0pt]{$B$}}}
\pgfcircle[fill]{\pgfxy(120.00,60.00)}{0.80mm}
\pgfcircle[stroke]{\pgfxy(120.00,60.00)}{0.80mm}
\pgfputat{\pgfxy(120.00,62.00)}{\pgfbox[bottom,left]{\fontsize{9.10}{10.93}\selectfont \makebox[0pt][r]{$m$}}}
\pgfputat{\pgfxy(118.00,48.00)}{\pgfbox[bottom,left]{\fontsize{9.10}{10.93}\selectfont \makebox[0pt][r]{$j$}}}
\pgfcircle[fill]{\pgfxy(130.00,60.00)}{0.80mm}
\pgfcircle[stroke]{\pgfxy(130.00,60.00)}{0.80mm}
\pgfputat{\pgfxy(136.00,63.00)}{\pgfbox[bottom,left]{\fontsize{9.10}{10.93}\selectfont \makebox[0pt]{$A$}}}
\pgfputat{\pgfxy(130.00,62.00)}{\pgfbox[bottom,left]{\fontsize{9.10}{10.93}\selectfont \makebox[0pt]{$s$}}}
\pgfellipse[stroke]{\pgfxy(133.50,63.49)}{\pgfxy(6.50,0.00)}{\pgfxy(0.00,4.51)}
\pgfellipse[stroke]{\pgfxy(133.50,53.49)}{\pgfxy(6.50,0.00)}{\pgfxy(0.00,4.51)}
\pgfcircle[fill]{\pgfxy(120.00,74.00)}{0.80mm}
\pgfcircle[stroke]{\pgfxy(120.00,74.00)}{0.80mm}
\pgfsetdash{{0.30mm}{0.50mm}}{0mm}\pgfmoveto{\pgfxy(120.00,74.00)}\pgflineto{\pgfxy(120.00,60.00)}\pgfstroke
\pgfputat{\pgfxy(118.00,73.00)}{\pgfbox[bottom,left]{\fontsize{9.10}{10.93}\selectfont \makebox[0pt][r]{i}}}
\end{pgfpicture}%
$$
    We get $\pi'=\varphi^{-1}(T'') \in \A_{n-2,i-2}$ where $T''\in \A_{n-2,i-2}$ is induced from $T'$ relabeling by $[n-2]$, where $i>k$. Relabeling $\pi'$ by $\set{1,\dots,k-2,k+1,\dots,n}$, denoted by $\pi''$, and define $\pi=(k,k-1,\pi'') \in \Ank$ where $\pi''_1=i$.
\item If $m=\infty$ or $m>s$ (Case~\ref{case:b1}), then we get the tree $T' \in \T_{n,k-1}$ modifying as follows: 
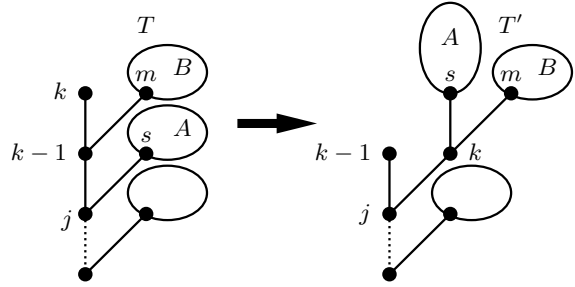
\label{case:A2}
    $$
    \centering
\begin{pgfpicture}{38.00mm}{29.20mm}{122.00mm}{70.00mm}
\pgfsetxvec{\pgfpoint{0.80mm}{0mm}}
\pgfsetyvec{\pgfpoint{0mm}{0.80mm}}
\color[rgb]{0,0,0}\pgfsetlinewidth{0.30mm}\pgfsetdash{}{0mm}
\pgfputat{\pgfxy(80.00,80.00)}{\pgfbox[bottom,left]{\fontsize{9.10}{10.93}\selectfont \makebox[0pt]{$T$}}}
\pgfsetdash{{0.30mm}{0.50mm}}{0mm}\pgfmoveto{\pgfxy(70.00,50.00)}\pgflineto{\pgfxy(70.00,40.00)}\pgfstroke
\pgfcircle[fill]{\pgfxy(70.00,50.00)}{0.80mm}
\pgfsetdash{}{0mm}\pgfcircle[stroke]{\pgfxy(70.00,50.00)}{0.80mm}
\pgfcircle[fill]{\pgfxy(70.00,40.00)}{0.80mm}
\pgfcircle[stroke]{\pgfxy(70.00,40.00)}{0.80mm}
\pgfmoveto{\pgfxy(70.00,50.00)}\pgflineto{\pgfxy(80.00,60.00)}\pgfstroke
\pgfcircle[fill]{\pgfxy(70.00,70.00)}{0.80mm}
\pgfcircle[stroke]{\pgfxy(70.00,70.00)}{0.80mm}
\pgfmoveto{\pgfxy(70.00,40.00)}\pgflineto{\pgfxy(80.00,50.00)}\pgfstroke
\pgfcircle[fill]{\pgfxy(80.00,50.00)}{0.80mm}
\pgfcircle[stroke]{\pgfxy(80.00,50.00)}{0.80mm}
\pgfputat{\pgfxy(67.00,69.00)}{\pgfbox[bottom,left]{\fontsize{9.10}{10.93}\selectfont \makebox[0pt][r]{$k$}}}
\pgfputat{\pgfxy(67.00,59.00)}{\pgfbox[bottom,left]{\fontsize{9.10}{10.93}\selectfont \makebox[0pt][r]{$k-1$}}}
\pgfcircle[fill]{\pgfxy(70.00,60.00)}{0.80mm}
\pgfcircle[stroke]{\pgfxy(70.00,60.00)}{0.80mm}
\pgfmoveto{\pgfxy(70.00,50.00)}\pgflineto{\pgfxy(70.00,60.00)}\pgfstroke
\pgfmoveto{\pgfxy(80.00,70.00)}\pgflineto{\pgfxy(70.00,60.00)}\pgfstroke
\pgfmoveto{\pgfxy(70.00,70.00)}\pgflineto{\pgfxy(70.00,60.00)}\pgfstroke
\pgfellipse[stroke]{\pgfxy(83.50,73.49)}{\pgfxy(6.50,0.00)}{\pgfxy(0.00,4.51)}
\pgfputat{\pgfxy(86.00,73.00)}{\pgfbox[bottom,left]{\fontsize{9.10}{10.93}\selectfont \makebox[0pt]{$B$}}}
\pgfcircle[fill]{\pgfxy(80.00,70.00)}{0.80mm}
\pgfcircle[stroke]{\pgfxy(80.00,70.00)}{0.80mm}
\pgfputat{\pgfxy(80.00,72.00)}{\pgfbox[bottom,left]{\fontsize{9.10}{10.93}\selectfont \makebox[0pt]{$m$}}}
\pgfputat{\pgfxy(68.00,48.00)}{\pgfbox[bottom,left]{\fontsize{9.10}{10.93}\selectfont \makebox[0pt][r]{$j$}}}
\pgfcircle[fill]{\pgfxy(80.00,60.00)}{0.80mm}
\pgfcircle[stroke]{\pgfxy(80.00,60.00)}{0.80mm}
\pgfputat{\pgfxy(86.00,63.00)}{\pgfbox[bottom,left]{\fontsize{9.10}{10.93}\selectfont \makebox[0pt]{$A$}}}
\pgfputat{\pgfxy(80.00,62.00)}{\pgfbox[bottom,left]{\fontsize{9.10}{10.93}\selectfont \makebox[0pt]{$s$}}}
\pgfellipse[stroke]{\pgfxy(83.50,63.49)}{\pgfxy(6.50,0.00)}{\pgfxy(0.00,4.51)}
\pgfellipse[stroke]{\pgfxy(83.50,53.49)}{\pgfxy(6.50,0.00)}{\pgfxy(0.00,4.51)}
\pgfsetlinewidth{1.20mm}\pgfmoveto{\pgfxy(95.00,65.00)}\pgflineto{\pgfxy(105.00,65.00)}\pgfstroke
\pgfmoveto{\pgfxy(105.00,65.00)}\pgflineto{\pgfxy(102.20,65.70)}\pgflineto{\pgfxy(102.20,64.30)}\pgflineto{\pgfxy(105.00,65.00)}\pgfclosepath\pgffill
\pgfmoveto{\pgfxy(105.00,65.00)}\pgflineto{\pgfxy(102.20,65.70)}\pgflineto{\pgfxy(102.20,64.30)}\pgflineto{\pgfxy(105.00,65.00)}\pgfclosepath\pgfstroke
\pgfputat{\pgfxy(140.00,80.00)}{\pgfbox[bottom,left]{\fontsize{9.10}{10.93}\selectfont \makebox[0pt]{$T'$}}}
\pgfputat{\pgfxy(117.00,59.00)}{\pgfbox[bottom,left]{\fontsize{9.10}{10.93}\selectfont \makebox[0pt][r]{$k-1$}}}
\pgfsetdash{{0.30mm}{0.50mm}}{0mm}\pgfsetlinewidth{0.30mm}\pgfmoveto{\pgfxy(120.00,50.00)}\pgflineto{\pgfxy(120.00,40.00)}\pgfstroke
\pgfcircle[fill]{\pgfxy(130.00,60.00)}{0.80mm}
\pgfsetdash{}{0mm}\pgfcircle[stroke]{\pgfxy(130.00,60.00)}{0.80mm}
\pgfcircle[fill]{\pgfxy(120.00,50.00)}{0.80mm}
\pgfcircle[stroke]{\pgfxy(120.00,50.00)}{0.80mm}
\pgfcircle[fill]{\pgfxy(120.00,40.00)}{0.80mm}
\pgfcircle[stroke]{\pgfxy(120.00,40.00)}{0.80mm}
\pgfmoveto{\pgfxy(130.00,60.00)}\pgflineto{\pgfxy(140.00,70.00)}\pgfstroke
\pgfmoveto{\pgfxy(120.00,40.00)}\pgflineto{\pgfxy(130.00,50.00)}\pgfstroke
\pgfcircle[fill]{\pgfxy(130.00,50.00)}{0.80mm}
\pgfcircle[stroke]{\pgfxy(130.00,50.00)}{0.80mm}
\pgfputat{\pgfxy(133.00,59.00)}{\pgfbox[bottom,left]{\fontsize{9.10}{10.93}\selectfont $k$}}
\pgfcircle[fill]{\pgfxy(120.00,60.00)}{0.80mm}
\pgfcircle[stroke]{\pgfxy(120.00,60.00)}{0.80mm}
\pgfmoveto{\pgfxy(120.00,50.00)}\pgflineto{\pgfxy(120.00,60.00)}\pgfstroke
\pgfputat{\pgfxy(117.00,49.00)}{\pgfbox[bottom,left]{\fontsize{9.10}{10.93}\selectfont \makebox[0pt][r]{$j$}}}
\pgfmoveto{\pgfxy(130.00,60.00)}\pgflineto{\pgfxy(120.00,50.00)}\pgfstroke
\pgfmoveto{\pgfxy(130.00,70.00)}\pgflineto{\pgfxy(130.00,60.00)}\pgfstroke
\pgfellipse[stroke]{\pgfxy(130.00,77.49)}{\pgfxy(5.00,0.00)}{\pgfxy(0.00,7.51)}
\pgfputat{\pgfxy(130.00,78.00)}{\pgfbox[bottom,left]{\fontsize{9.10}{10.93}\selectfont \makebox[0pt]{$A$}}}
\pgfcircle[fill]{\pgfxy(130.00,70.00)}{0.80mm}
\pgfcircle[stroke]{\pgfxy(130.00,70.00)}{0.80mm}
\pgfputat{\pgfxy(140.00,72.00)}{\pgfbox[bottom,left]{\fontsize{9.10}{10.93}\selectfont \makebox[0pt]{$m$}}}
\pgfcircle[fill]{\pgfxy(140.00,70.00)}{0.80mm}
\pgfcircle[stroke]{\pgfxy(140.00,70.00)}{0.80mm}
\pgfputat{\pgfxy(146.00,73.00)}{\pgfbox[bottom,left]{\fontsize{9.10}{10.93}\selectfont \makebox[0pt]{$B$}}}
\pgfputat{\pgfxy(130.00,72.00)}{\pgfbox[bottom,left]{\fontsize{9.10}{10.93}\selectfont \makebox[0pt]{$s$}}}
\pgfellipse[stroke]{\pgfxy(143.50,73.49)}{\pgfxy(6.50,0.00)}{\pgfxy(0.00,4.51)}
\pgfellipse[stroke]{\pgfxy(133.50,53.49)}{\pgfxy(6.50,0.00)}{\pgfxy(0.00,4.51)}
\end{pgfpicture}%
$$
    Define $\pi'=\varphi^{-1}(T') \in \A_{n,k-1}$ and $\pi=(k-1~k)\pi' \in \A_{n,k}$ (exchange $k-1$ and $k$ in $\pi'$).
\end{enumerate}
\item If $k-1$ is not a parent of $k$ in $T$ (Case~\ref{case:b2}), then we get the tree $T' \in \Tnk$ exchanging the labels $k-1$ and $k$ in $T$. 
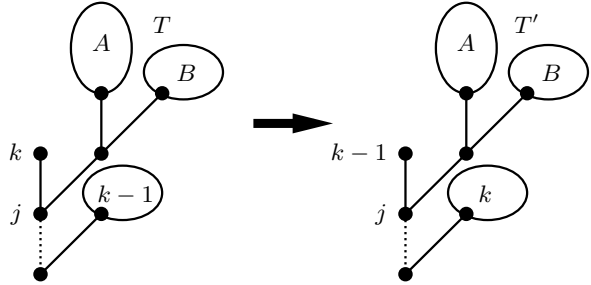
\label{case:B}
    $$
    \centering
\begin{pgfpicture}{38.00mm}{29.20mm}{122.00mm}{70.00mm}
\pgfsetxvec{\pgfpoint{0.80mm}{0mm}}
\pgfsetyvec{\pgfpoint{0mm}{0.80mm}}
\color[rgb]{0,0,0}\pgfsetlinewidth{0.30mm}\pgfsetdash{}{0mm}
\pgfsetlinewidth{1.20mm}\pgfmoveto{\pgfxy(95.00,65.00)}\pgflineto{\pgfxy(105.00,65.00)}\pgfstroke
\pgfmoveto{\pgfxy(105.00,65.00)}\pgflineto{\pgfxy(102.20,65.70)}\pgflineto{\pgfxy(102.20,64.30)}\pgflineto{\pgfxy(105.00,65.00)}\pgfclosepath\pgffill
\pgfmoveto{\pgfxy(105.00,65.00)}\pgflineto{\pgfxy(102.20,65.70)}\pgflineto{\pgfxy(102.20,64.30)}\pgflineto{\pgfxy(105.00,65.00)}\pgfclosepath\pgfstroke
\pgfputat{\pgfxy(140.00,80.00)}{\pgfbox[bottom,left]{\fontsize{9.10}{10.93}\selectfont \makebox[0pt]{$T'$}}}
\pgfputat{\pgfxy(117.00,59.00)}{\pgfbox[bottom,left]{\fontsize{9.10}{10.93}\selectfont \makebox[0pt][r]{$k-1$}}}
\pgfsetdash{{0.30mm}{0.50mm}}{0mm}\pgfsetlinewidth{0.30mm}\pgfmoveto{\pgfxy(120.00,50.00)}\pgflineto{\pgfxy(120.00,40.00)}\pgfstroke
\pgfcircle[fill]{\pgfxy(130.00,60.00)}{0.80mm}
\pgfsetdash{}{0mm}\pgfcircle[stroke]{\pgfxy(130.00,60.00)}{0.80mm}
\pgfcircle[fill]{\pgfxy(120.00,50.00)}{0.80mm}
\pgfcircle[stroke]{\pgfxy(120.00,50.00)}{0.80mm}
\pgfcircle[fill]{\pgfxy(120.00,40.00)}{0.80mm}
\pgfcircle[stroke]{\pgfxy(120.00,40.00)}{0.80mm}
\pgfmoveto{\pgfxy(130.00,60.00)}\pgflineto{\pgfxy(140.00,70.00)}\pgfstroke
\pgfmoveto{\pgfxy(120.00,40.00)}\pgflineto{\pgfxy(130.00,50.00)}\pgfstroke
\pgfcircle[fill]{\pgfxy(130.00,50.00)}{0.80mm}
\pgfcircle[stroke]{\pgfxy(130.00,50.00)}{0.80mm}
\pgfputat{\pgfxy(132.00,52.00)}{\pgfbox[bottom,left]{\fontsize{9.10}{10.93}\selectfont $k$}}
\pgfcircle[fill]{\pgfxy(120.00,60.00)}{0.80mm}
\pgfcircle[stroke]{\pgfxy(120.00,60.00)}{0.80mm}
\pgfmoveto{\pgfxy(120.00,50.00)}\pgflineto{\pgfxy(120.00,60.00)}\pgfstroke
\pgfputat{\pgfxy(117.00,49.00)}{\pgfbox[bottom,left]{\fontsize{9.10}{10.93}\selectfont \makebox[0pt][r]{$j$}}}
\pgfmoveto{\pgfxy(130.00,60.00)}\pgflineto{\pgfxy(120.00,50.00)}\pgfstroke
\pgfmoveto{\pgfxy(130.00,70.00)}\pgflineto{\pgfxy(130.00,60.00)}\pgfstroke
\pgfellipse[stroke]{\pgfxy(130.00,77.49)}{\pgfxy(5.00,0.00)}{\pgfxy(0.00,7.51)}
\pgfputat{\pgfxy(130.00,77.00)}{\pgfbox[bottom,left]{\fontsize{9.10}{10.93}\selectfont \makebox[0pt]{$A$}}}
\pgfcircle[fill]{\pgfxy(130.00,70.00)}{0.80mm}
\pgfcircle[stroke]{\pgfxy(130.00,70.00)}{0.80mm}
\pgfcircle[fill]{\pgfxy(140.00,70.00)}{0.80mm}
\pgfcircle[stroke]{\pgfxy(140.00,70.00)}{0.80mm}
\pgfputat{\pgfxy(144.00,72.00)}{\pgfbox[bottom,left]{\fontsize{9.10}{10.93}\selectfont \makebox[0pt]{$B$}}}
\pgfellipse[stroke]{\pgfxy(143.50,73.49)}{\pgfxy(6.50,0.00)}{\pgfxy(0.00,4.51)}
\pgfellipse[stroke]{\pgfxy(133.50,53.49)}{\pgfxy(6.50,0.00)}{\pgfxy(0.00,4.51)}
\pgfputat{\pgfxy(80.00,80.00)}{\pgfbox[bottom,left]{\fontsize{9.10}{10.93}\selectfont \makebox[0pt]{$T$}}}
\pgfputat{\pgfxy(57.00,59.00)}{\pgfbox[bottom,left]{\fontsize{9.10}{10.93}\selectfont \makebox[0pt][r]{$k$}}}
\pgfsetdash{{0.30mm}{0.50mm}}{0mm}\pgfmoveto{\pgfxy(60.00,50.00)}\pgflineto{\pgfxy(60.00,40.00)}\pgfstroke
\pgfcircle[fill]{\pgfxy(70.00,60.00)}{0.80mm}
\pgfsetdash{}{0mm}\pgfcircle[stroke]{\pgfxy(70.00,60.00)}{0.80mm}
\pgfcircle[fill]{\pgfxy(60.00,50.00)}{0.80mm}
\pgfcircle[stroke]{\pgfxy(60.00,50.00)}{0.80mm}
\pgfcircle[fill]{\pgfxy(60.00,40.00)}{0.80mm}
\pgfcircle[stroke]{\pgfxy(60.00,40.00)}{0.80mm}
\pgfmoveto{\pgfxy(70.00,60.00)}\pgflineto{\pgfxy(80.00,70.00)}\pgfstroke
\pgfmoveto{\pgfxy(60.00,40.00)}\pgflineto{\pgfxy(70.00,50.00)}\pgfstroke
\pgfcircle[fill]{\pgfxy(70.00,50.00)}{0.80mm}
\pgfcircle[stroke]{\pgfxy(70.00,50.00)}{0.80mm}
\pgfputat{\pgfxy(74.00,52.00)}{\pgfbox[bottom,left]{\fontsize{9.10}{10.93}\selectfont \makebox[0pt]{$k-1$}}}
\pgfcircle[fill]{\pgfxy(60.00,60.00)}{0.80mm}
\pgfcircle[stroke]{\pgfxy(60.00,60.00)}{0.80mm}
\pgfmoveto{\pgfxy(60.00,50.00)}\pgflineto{\pgfxy(60.00,60.00)}\pgfstroke
\pgfputat{\pgfxy(57.00,49.00)}{\pgfbox[bottom,left]{\fontsize{9.10}{10.93}\selectfont \makebox[0pt][r]{$j$}}}
\pgfmoveto{\pgfxy(70.00,60.00)}\pgflineto{\pgfxy(60.00,50.00)}\pgfstroke
\pgfmoveto{\pgfxy(70.00,70.00)}\pgflineto{\pgfxy(70.00,60.00)}\pgfstroke
\pgfellipse[stroke]{\pgfxy(70.00,77.49)}{\pgfxy(5.00,0.00)}{\pgfxy(0.00,7.51)}
\pgfcircle[fill]{\pgfxy(70.00,70.00)}{0.80mm}
\pgfcircle[stroke]{\pgfxy(70.00,70.00)}{0.80mm}
\pgfcircle[fill]{\pgfxy(80.00,70.00)}{0.80mm}
\pgfcircle[stroke]{\pgfxy(80.00,70.00)}{0.80mm}
\pgfellipse[stroke]{\pgfxy(83.50,73.49)}{\pgfxy(6.50,0.00)}{\pgfxy(0.00,4.51)}
\pgfellipse[stroke]{\pgfxy(73.50,53.49)}{\pgfxy(6.50,0.00)}{\pgfxy(0.00,4.51)}
\pgfputat{\pgfxy(70.00,77.00)}{\pgfbox[bottom,left]{\fontsize{9.10}{10.93}\selectfont \makebox[0pt]{$A$}}}
\pgfputat{\pgfxy(84.00,72.00)}{\pgfbox[bottom,left]{\fontsize{9.10}{10.93}\selectfont \makebox[0pt]{$B$}}}
\end{pgfpicture}%
$$
    Define $\pi'=\varphi^{-1}(T') \in \A_{n,k-1}$ and $\pi=(k-1~k)\pi' \in \A_{n,k}$ (exchange $k-1$ and $k$ in $\pi'$).
\end{enumerate}
\end{proof}

\begin{rmk}
By considering a description of $\varphi$, given a elements $\pi \in \Ank$, $\pi$ in Case~\eqref{case:a} can be induced from $\cup_{i=k+1}^{n}\A_{n-2,i-2}$ and $\pi$ in Case~\eqref{case:b} from $\A_{n,k-1}$.
It yields that two following sets are {\em isomorphic} $$\Ank \simeq \A_{n,k-1} \cup (\cup_{i=k+1}^{n}\A_{n-2,i-2}),$$
for $2\le k \le n$.
We get the recurrence relation $$a_{n,k}=a_{n,k-1}+\sum_{i=k+1}^{n}a_{n-2,i-2},$$ where $a_{n,k}$ is the cardinality of the set $\Ank$.
Indeed, we are able to generalize this recurrence relation. Let $a_{n,k}(q,p)=\sum_{\sigma \in \Ank} q^{\inv(\pi)}p^{\312(\pi)}$. We have also $$a_{n,k}(q,p)=q~p~a_{n,k-1}(q,p)+q^{2k-3}\sum_{i=k+1}^{n}a_{n-2,i-2}(q,p),$$
for $2\le k \le n$.

Similarly, by considering a description of $\varphi^{-1}$, given a elements $T \in \Tnk$, $T$ in Case~\eqref{case:A1} can be induced from $\cup_{i=k+1}^{n}\T_{n-2,i-2}$ and $T$ in Case~\eqref{case:A2} and Case~\eqref{case:B} from $\T_{n,k-1}$.
It yields $$\Tnk \simeq \T_{n,k-1} \cup (\cup_{i=k+1}^{n}\T_{n-2,i-2}).$$
We get the recurrence relation $$t_{n,k}=t_{n,k-1}+\sum_{i=k+1}^{n}t_{n-2,i-2},$$ where $t_{n,k}$ is the cardinality of the set $\Tnk$.
\end{rmk}

\begin{ex} The following is the table of $\varphi$ on $\A_4$.
$$
\begin{tabular}{c||c|c|c|c|c}
$\pi$ & 2143 & 3142 & 3241 & 4132 & 4231
\\
\hline
$\pi_1$ & 2 & 3 & 3 & 4 & 4
\\
\hline
$\inv(\pi)$ & 2 & 3 & 3 & 4 & 4
\\
\hline
$\312(\pi)$ & 0 & 1 & 0 & 2 & 1
\\
\hline
\hline
$\varphi(\pi)$&
$
\centering
\begin{pgfpicture}{2.20mm}{3.70mm}{15.30mm}{17.50mm}
\pgfsetxvec{\pgfpoint{0.70mm}{0mm}}
\pgfsetyvec{\pgfpoint{0mm}{0.70mm}}
\color[rgb]{0,0,0}\pgfsetlinewidth{0.30mm}\pgfsetdash{}{0mm}
\pgfcircle[fill]{\pgfxy(10.00,10.00)}{0.70mm}
\pgfcircle[stroke]{\pgfxy(10.00,10.00)}{0.70mm}
\pgfcircle[fill]{\pgfxy(10.00,15.00)}{0.70mm}
\pgfcircle[stroke]{\pgfxy(10.00,15.00)}{0.70mm}
\pgfcircle[fill]{\pgfxy(15.00,15.00)}{0.70mm}
\pgfcircle[stroke]{\pgfxy(15.00,15.00)}{0.70mm}
\pgfcircle[fill]{\pgfxy(15.00,20.00)}{0.70mm}
\pgfcircle[stroke]{\pgfxy(15.00,20.00)}{0.70mm}
\pgfmoveto{\pgfxy(10.00,15.00)}\pgflineto{\pgfxy(10.00,10.00)}\pgfstroke
\pgfmoveto{\pgfxy(10.00,10.00)}\pgflineto{\pgfxy(15.00,15.00)}\pgfstroke
\pgfmoveto{\pgfxy(15.00,20.00)}\pgflineto{\pgfxy(15.00,15.00)}\pgfstroke
\pgfputat{\pgfxy(8.00,14.00)}{\pgfbox[bottom,left]{\fontsize{7.97}{9.56}\selectfont \makebox[0pt][r]{2}}}
\pgfputat{\pgfxy(8.00,9.00)}{\pgfbox[bottom,left]{\fontsize{7.97}{9.56}\selectfont \makebox[0pt][r]{1}}}
\pgfputat{\pgfxy(17.00,14.00)}{\pgfbox[bottom,left]{\fontsize{7.97}{9.56}\selectfont 3}}
\pgfputat{\pgfxy(17.00,19.00)}{\pgfbox[bottom,left]{\fontsize{7.97}{9.56}\selectfont 4}}
\end{pgfpicture}%
$ &
$
\centering
\begin{pgfpicture}{2.20mm}{3.70mm}{15.30mm}{17.50mm}
\pgfsetxvec{\pgfpoint{0.70mm}{0mm}}
\pgfsetyvec{\pgfpoint{0mm}{0.70mm}}
\color[rgb]{0,0,0}\pgfsetlinewidth{0.30mm}\pgfsetdash{}{0mm}
\pgfcircle[fill]{\pgfxy(10.00,10.00)}{0.70mm}
\pgfcircle[stroke]{\pgfxy(10.00,10.00)}{0.70mm}
\pgfcircle[fill]{\pgfxy(10.00,15.00)}{0.70mm}
\pgfcircle[stroke]{\pgfxy(10.00,15.00)}{0.70mm}
\pgfcircle[fill]{\pgfxy(10.00,20.00)}{0.70mm}
\pgfcircle[stroke]{\pgfxy(10.00,20.00)}{0.70mm}
\pgfcircle[fill]{\pgfxy(15.00,15.00)}{0.70mm}
\pgfcircle[stroke]{\pgfxy(15.00,15.00)}{0.70mm}
\pgfmoveto{\pgfxy(10.00,15.00)}\pgflineto{\pgfxy(10.00,10.00)}\pgfstroke
\pgfmoveto{\pgfxy(10.00,10.00)}\pgflineto{\pgfxy(15.00,15.00)}\pgfstroke
\pgfmoveto{\pgfxy(10.00,20.00)}\pgflineto{\pgfxy(10.00,15.00)}\pgfstroke
\pgfputat{\pgfxy(8.00,14.00)}{\pgfbox[bottom,left]{\fontsize{7.97}{9.56}\selectfont \makebox[0pt][r]{2}}}
\pgfputat{\pgfxy(8.00,9.00)}{\pgfbox[bottom,left]{\fontsize{7.97}{9.56}\selectfont \makebox[0pt][r]{1}}}
\pgfputat{\pgfxy(8.00,19.00)}{\pgfbox[bottom,left]{\fontsize{7.97}{9.56}\selectfont \makebox[0pt][r]{3}}}
\pgfputat{\pgfxy(17.00,14.00)}{\pgfbox[bottom,left]{\fontsize{7.97}{9.56}\selectfont 4}}
\end{pgfpicture}%
$ &
$
\centering
\begin{pgfpicture}{2.20mm}{3.70mm}{15.30mm}{17.50mm}
\pgfsetxvec{\pgfpoint{0.70mm}{0mm}}
\pgfsetyvec{\pgfpoint{0mm}{0.70mm}}
\color[rgb]{0,0,0}\pgfsetlinewidth{0.30mm}\pgfsetdash{}{0mm}
\pgfcircle[fill]{\pgfxy(10.00,10.00)}{0.70mm}
\pgfcircle[stroke]{\pgfxy(10.00,10.00)}{0.70mm}
\pgfcircle[fill]{\pgfxy(10.00,15.00)}{0.70mm}
\pgfcircle[stroke]{\pgfxy(10.00,15.00)}{0.70mm}
\pgfcircle[fill]{\pgfxy(10.00,20.00)}{0.70mm}
\pgfcircle[stroke]{\pgfxy(10.00,20.00)}{0.70mm}
\pgfcircle[fill]{\pgfxy(15.00,20.00)}{0.70mm}
\pgfcircle[stroke]{\pgfxy(15.00,20.00)}{0.70mm}
\pgfmoveto{\pgfxy(10.00,15.00)}\pgflineto{\pgfxy(10.00,10.00)}\pgfstroke
\pgfmoveto{\pgfxy(10.00,15.00)}\pgflineto{\pgfxy(15.00,20.00)}\pgfstroke
\pgfmoveto{\pgfxy(10.00,20.00)}\pgflineto{\pgfxy(10.00,15.00)}\pgfstroke
\pgfputat{\pgfxy(8.00,14.00)}{\pgfbox[bottom,left]{\fontsize{7.97}{9.56}\selectfont \makebox[0pt][r]{2}}}
\pgfputat{\pgfxy(8.00,9.00)}{\pgfbox[bottom,left]{\fontsize{7.97}{9.56}\selectfont \makebox[0pt][r]{1}}}
\pgfputat{\pgfxy(8.00,19.00)}{\pgfbox[bottom,left]{\fontsize{7.97}{9.56}\selectfont \makebox[0pt][r]{3}}}
\pgfputat{\pgfxy(17.00,19.00)}{\pgfbox[bottom,left]{\fontsize{7.97}{9.56}\selectfont 4}}
\end{pgfpicture}%
$ &
$
\centering
\begin{pgfpicture}{16.20mm}{3.70mm}{29.30mm}{17.50mm}
\pgfsetxvec{\pgfpoint{0.70mm}{0mm}}
\pgfsetyvec{\pgfpoint{0mm}{0.70mm}}
\color[rgb]{0,0,0}\pgfsetlinewidth{0.30mm}\pgfsetdash{}{0mm}
\pgfcircle[fill]{\pgfxy(30.00,10.00)}{0.70mm}
\pgfcircle[stroke]{\pgfxy(30.00,10.00)}{0.70mm}
\pgfcircle[fill]{\pgfxy(30.00,15.00)}{0.70mm}
\pgfcircle[stroke]{\pgfxy(30.00,15.00)}{0.70mm}
\pgfcircle[fill]{\pgfxy(30.00,20.00)}{0.70mm}
\pgfcircle[stroke]{\pgfxy(30.00,20.00)}{0.70mm}
\pgfcircle[fill]{\pgfxy(35.00,15.00)}{0.70mm}
\pgfcircle[stroke]{\pgfxy(35.00,15.00)}{0.70mm}
\pgfmoveto{\pgfxy(30.00,15.00)}\pgflineto{\pgfxy(30.00,10.00)}\pgfstroke
\pgfmoveto{\pgfxy(30.00,10.00)}\pgflineto{\pgfxy(35.00,15.00)}\pgfstroke
\pgfmoveto{\pgfxy(30.00,15.00)}\pgflineto{\pgfxy(30.00,20.00)}\pgfstroke
\pgfputat{\pgfxy(28.00,14.00)}{\pgfbox[bottom,left]{\fontsize{7.97}{9.56}\selectfont \makebox[0pt][r]{2}}}
\pgfputat{\pgfxy(28.00,9.00)}{\pgfbox[bottom,left]{\fontsize{7.97}{9.56}\selectfont \makebox[0pt][r]{1}}}
\pgfputat{\pgfxy(28.00,19.00)}{\pgfbox[bottom,left]{\fontsize{7.97}{9.56}\selectfont \makebox[0pt][r]{4}}}
\pgfputat{\pgfxy(37.00,14.00)}{\pgfbox[bottom,left]{\fontsize{7.97}{9.56}\selectfont 3}}
\end{pgfpicture}%
$ &
$
\centering
\begin{pgfpicture}{16.20mm}{3.70mm}{23.70mm}{20.65mm}
\pgfsetxvec{\pgfpoint{0.70mm}{0mm}}
\pgfsetyvec{\pgfpoint{0mm}{0.70mm}}
\color[rgb]{0,0,0}\pgfsetlinewidth{0.30mm}\pgfsetdash{}{0mm}
\pgfcircle[fill]{\pgfxy(30.00,10.00)}{0.70mm}
\pgfcircle[stroke]{\pgfxy(30.00,10.00)}{0.70mm}
\pgfcircle[fill]{\pgfxy(30.00,15.00)}{0.70mm}
\pgfcircle[stroke]{\pgfxy(30.00,15.00)}{0.70mm}
\pgfcircle[fill]{\pgfxy(30.00,20.00)}{0.70mm}
\pgfcircle[stroke]{\pgfxy(30.00,20.00)}{0.70mm}
\pgfcircle[fill]{\pgfxy(30.00,25.00)}{0.70mm}
\pgfcircle[stroke]{\pgfxy(30.00,25.00)}{0.70mm}
\pgfmoveto{\pgfxy(30.00,15.00)}\pgflineto{\pgfxy(30.00,10.00)}\pgfstroke
\pgfputat{\pgfxy(28.00,14.00)}{\pgfbox[bottom,left]{\fontsize{7.97}{9.56}\selectfont \makebox[0pt][r]{2}}}
\pgfputat{\pgfxy(28.00,9.00)}{\pgfbox[bottom,left]{\fontsize{7.97}{9.56}\selectfont \makebox[0pt][r]{1}}}
\pgfputat{\pgfxy(28.00,18.50)}{\pgfbox[bottom,left]{\fontsize{7.97}{9.56}\selectfont \makebox[0pt][r]{3}}}
\pgfputat{\pgfxy(28.00,23.50)}{\pgfbox[bottom,left]{\fontsize{7.97}{9.56}\selectfont \makebox[0pt][r]{4}}}
\pgfmoveto{\pgfxy(30.00,20.00)}\pgflineto{\pgfxy(30.00,15.00)}\pgfstroke
\pgfmoveto{\pgfxy(30.00,25.00)}\pgflineto{\pgfxy(30.00,20.00)}\pgfstroke
\end{pgfpicture}%
$
\\
\hline
$p(\varphi(\pi))$ & 2 & 3 & 3 & 4 & 4
\\
\end{tabular}
$$
\end{ex}


\providecommand{\bysame}{\leavevmode\hbox to3em{\hrulefill}\thinspace}
\providecommand{\href}[2]{#2}

\end{document}